\documentclass{amsart}
\usepackage[utf8]{inputenc}

\usepackage{geometry}
\usepackage{setspace}
\usepackage{graphicx}
\usepackage{amssymb}
\usepackage{mathrsfs}
\usepackage{amsthm}
\usepackage{amsmath}
\usepackage{color}
\usepackage{caption}
\usepackage{comment}
\usepackage{subcaption}
\usepackage{mathtools}
\usepackage{enumerate}
\usepackage{algorithm,algorithmic}

\usepackage{hyperref}

\usepackage[sorting=anyt]{biblatex}
\newtheorem{theorem}{Theorem}[section]
\newtheorem{corollary}{Corollary}[theorem]

\newtheorem{lemma}[theorem]{Lemma}
\newtheorem{assumption}[theorem]{Assumption}

\theoremstyle{remark}

\theoremstyle{definition}

\makeatletter
\newcommand{\myitem}[1]{%
\item[#1]\protected@edef\@currentlabel{#1}%
}
\makeatother

\title[ ]{Numerical approximation of linear parabolic evolution equations revisited}
\author[Ø.~S.~Auestad]{Øyvind S. Auestad}
\address{\newline Department of Mathematical Sciences Norwegian University of Science and Technology, \newline 7034 Trondheim, Norway.} \email{oyvinau@ntnu.no}

\begin{document}

\begin{abstract}
    We obtain rates of convergence of numerical approximations of abstract linear parabolic evolution equations in Banach spaces. Our estimates extend known results from the literature of finite element approximations of parabolic equations to more general equations and numerical approximation methods. As an example, we consider parabolic equations on surfaces and surface finite element approximations. 
\end{abstract}

\date{\today}
\thanks{The research of the author was supported by Grant No. 325114 of the Norwegian Research Council.}

\maketitle

\section{introduction}

We derive semidiscrete and fully discrete error estimates of numerical approximations of abstract linear evolution equations,
\begin{align}\label{eq:linear-evolution-equation}
    \dot{u} = -A u, \quad u(0) = x \in X,
\end{align}
where $X$ is a Banach space and $-A$ is the generator of an analytic semigroup on $X$. The main result of this paper is Theorem \ref{theorem:1}, which is an abstract semidiscrete error estimate that generalize established estimates from the literature of finite element approximations of parabolic equations. In particular, when $A$ is a second order elliptic operator and $X = L^2(\Omega)$ for some regular domain $\Omega$, one can show that under appropriate conditions there is $\lambda \geq 0$ and $C, \epsilon > 0$ such that
\begin{align*}
    \Vert u(t) - u_h(t) \Vert_X \leq C e^{(\lambda - \epsilon) t} t^{-\theta / 2 + \rho / 2} h^{\theta} \Vert (\lambda + A)^{\rho / 2} x \Vert_X, \quad \theta \in [0,2], \ \rho \in [-1, \theta] \cap [-2 + \theta, \theta],
\end{align*}
where $0 < h < 1$ and $u_h$ is the ordinary piecewise linear Galerkin approximation of $u$. In the case of $A$ selfadjoint, the non-smooth data estimate $\theta = 2, \rho = 0$ and the smooth data estimate $\theta = 2, \rho = 2$ follows by the results of \cite{1974-helfrich}. In the more general case of $A$ not necessarily selfadjoint, \cite{1976-fujita-mizutani} showed the non-smooth data estimate, smooth data estimates with $\theta = 1, \rho = 1$ and $\theta = 2, \rho = 1$, in addition to similar estimates in the $H^1(\Omega)$-norm, while \cite{2013-tambue} showed the smooth data estimate $\theta = 2, \rho = 2$. Finally, non-smooth data estimates with $\rho \leq 0$ were shown in the case of $A$ selfadjoint in \cite{2014-kruse}.

We extend these error estimates to a broader class of equations and numerical methods. First, in Theorem \ref{theorem:1}, which holds for evolution equations in Banach spaces where $-A$ generates an analytic semigroup, and its numerical approximation is only required to be uniformly sectorial (see Assumption \ref{assumption:abstract-model} for precise conditions). Then, in Lemma \ref{lemma:1} and \ref{lemma:2}, which under similar conditions covers semidiscrete error estimates in norms related to fractional powers of the (possibly shifted) operator $A$, and finally, in Lemma \ref{lemma:negative-rho} which is an error estimate accommodating less regular data than Theorem \ref{theorem:1}. As an example, we show in Corollary \ref{cor:semidiscrete-numerical-example} that this covers equations involving elliptic operators on surfaces, and the surface finite element approximation proposed in \cite{1988-dziuk}. A key component of all proofs are the interpolation inequalities of Lemma \ref{lemma:interpolation-inequality}, which includes a Heinz--Kato type inequality for operators on Banach spaces. 

Further, we derive similar fully discrete error estimates by combining our semidiscerete approximation with a backward Euler approximation in time, and Theorem \ref{theorem:fully-discrete-semigroup-approximation-2} and \ref{theorem:fully-discrete-semigroup-approximation-4} describes rates of convergence in this case. In particular, the inequalities of Lemma \ref{lemma:fully-discrete-inequalities-1}, \ref{lemma:fully-discrete-inequalities-2} and \ref{lemma:fully-discrete-inequalities-3} extend Theorem 9.1, 9.2 and 9.3 in \cite{thomee} to equations involving operators that are not necessarily strictly positive, and contains an additional factor ensuring exponential decay in time whenever $A$ is strictly positive. In Corollary \ref{cor:fully-discrete-numerical-example} we show that this covers elliptic operators on surfaces, and surface finite element approximations combined with backward Euler.  

The motivation for studying parameterized estimates of the form of Theorem \ref{theorem:1}, \ref{theorem:fully-discrete-semigroup-approximation-2} and \ref{theorem:fully-discrete-semigroup-approximation-4}---where the user can trade off regularity of the initial condition against either a singularity at $t = 0$ or rate of convergence---is their importance in deriving optimal rates of convergence of numerical approximations of stochastic evolution equations. All results mentioned in this introduction turn out to be essential components of the error analysis in \cite{2024-fem-auestad} and \cite{2024-sfem-auestad}. 
\section{Semidiscrete and fully discerete error estimates}

Throughout the paper we use the following notation: for a Banach space $X$, we denote by $L(X)$ the Banach space of bounded linear operators on $X$ with its usual norm, while for a linear operator $A$, we denote its spectrum by $\sigma(A)$. We also denote by $C$ a generic constant, which may change from line to line. Parameter dependence of $C$ is omitted if not relevant. We will frequently (e.g. in Lemma \ref{lemma:interpolation-inequality}, condition \ref{cond:negative-rho} and \ref{cond:abstract-solution-operator-3}) and without further explanation consider operators $B \in L(X)$, with the property that $B A \in L(X)$ for some densely defined and possibly unbounded linear $A : D(A) \to X$, with a bounded inverse. By this we understand that $B$ extends to a bounded linear operator $D(A^{-1}) \to X$, where we understand $D(A^{-1})$ as the completion of $X$ using the $\Vert A^{-1} \cdot \Vert_X$-norm.  

\subsection{Semidiscrete error estimate}

In this section we study \eqref{eq:linear-evolution-equation} and its semidiscrete numerical approximation,
\begin{align}\label{eq:discrete-linear-evolution}
     \dot{u}_h  = -A_h u_h, \quad u_h(0) = \pi_h x \in X_h \subseteq X,
\end{align}
under the conditions of Assumption \ref{assumption:abstract-model}.

\begin{assumption}\label{assumption:abstract-model} 
\hfill
\begin{description}
    \myitem{(A1)} $X$ is a Banach space with norm $\Vert \cdot \Vert_X$, and $X_h \subseteq X, \ 0 < h < 1$, are closed linear spaces. \medskip
    \myitem{(A2)}\label{cond:pih} $\pi_h : X \to X_h$ is bounded and linear, $\Vert \pi_h \Vert_{L(X)}$ is independent of $h$, and $\pi_h x = x$ for any $x \in X_h$. \medskip
    \myitem{(A3)}\label{cond:sectorial} $A : D(A) \subseteq X \to X$ is a closed and densely defined linear operator, and it is sectorial in the sense that: there is $\lambda \geq 0$, such that
    \begin{align*}
         \sigma(\lambda + A) \subseteq \Sigma_{\delta} := \{ z \in \mathbb{C} \setminus \{ 0 \}, \ \vert \mathrm{arg}(z) \vert < \delta \},
    \end{align*}
    for some $\delta \in (0, \pi / 2)$, and for some $M > 0$,
    \begin{align*} 
        \Vert (z - \lambda - A)^{-1} x \Vert_X \leq \frac{M}{\vert z \vert} \Vert x \Vert_X, \quad \text{for any } x \in X \text{ and any } z \notin \Sigma_{\delta}.
    \end{align*}
    \myitem{(A4)}\label{cond:discrete-sectorial} $A_h : X_h \to X_h$ are linear operators, and are sectorial (uniformly in $h$), in the sense that: for the same $\lambda, \delta, M$, as in \ref{cond:sectorial}, and for any $0 < h < 1$,
    \begin{align*}
         \sigma(\lambda + A_h) \subseteq \Sigma_{\delta} \quad \text{ and } \quad
        \Vert (z - \lambda - A_h)^{-1} x \Vert_X \leq \frac{M}{\vert z \vert} \Vert x \Vert_X, \quad \text{for any } x \in X_h \text{ and any } z \notin \Sigma_{\delta}.
    \end{align*}
    \myitem{(A5)}\label{cond:abstract-solution-operator} For some $r > 0$, the following estimate holds, 
    \begin{align*}
        \Vert ((\lambda + A)^{-1} - (\lambda + A_h)^{-1} \pi_h) x \Vert_X \leq C h^r \Vert x \Vert_X, \quad x \in X.
    \end{align*}
\end{description}
\end{assumption}

$-A$ and $-A_h$ defined as in \ref{cond:sectorial} and \ref{cond:discrete-sectorial} are the generators of analytic semigroups on $X$ and $X_h$, respectively. In what follows these will be denoted by $S(\cdot)$ and $S_h(\cdot)$. For any $t \geq 0$ we define,
\begin{align*}
    u(t) := S(t) x, \quad u_h(t) := S_h(t) \pi_h x,
\end{align*}
as the solutions to \eqref{eq:linear-evolution-equation} and \eqref{eq:discrete-linear-evolution} respectively. It is well known that with $u, u_h$ above, \eqref{eq:linear-evolution-equation} and \eqref{eq:discrete-linear-evolution} holds in the sense of the limit in the time derivative being taken in the norm topology on $X$. The stronger result holds, 
\begin{align*}
    \lim_{\Delta t \to 0} \frac{S(t+\Delta t) - S(t)}{\Delta t} = -A S(t),
\end{align*}
with the limit being taken in the uniform operator topology, since $S(\cdot) : \Sigma_{\pi / 2 - \delta} \to L(X)$, defined in \eqref{eq:residue-theorem-semigroup}, is analytic (see e.g. Section 2.5.1 in \cite{yagi}). 

The condition \ref{cond:abstract-solution-operator} is interpreted as the error in the numerical approximation of the corresponding stationary problem $(\lambda + A) u = x$, where the numerical approximation is given by $(\lambda + A_h) u_h = \pi_h x$. When we can associate to $A$ a sesquilinear form, $a : V \times V \to \mathbb{C}$, where $V,H$ are Hilbert spaces with $V \subseteq H$ densely and continuously (see Lemma \ref{lemma:variational-semigroup}), we think of $\lambda \geq 0$ as some constant ensuring that ``the shifted sesquilinear form", $a(\cdot, \cdot) + \lambda (\cdot, \cdot)_H$, is coercive on $V$. Elliptic operators with Gårding's inequality is the canonical example. More generally, one can consider elliptic operators on surfaces with a Ladyzhenskaya inequality (as described in Section 3 in \cite{2013-elliott}). In Theorem \ref{theorem:1} we will see that even though $\lambda$ is needed for the stationary problem above to have a solution, it only influences the error in approximating the corresponding Cauchy problem \eqref{eq:linear-evolution-equation} through a factor that possibly grows or decays exponentially with time.

The following theorem is the main result of this paper, and is an extension of Theorem 3.5 in \cite{thomee} to more general equations and numerical approximation techniques.
\begin{theorem}\label{theorem:1}
Under Assumption \ref{assumption:abstract-model}, there is $C, \epsilon > 0$ such that
\begin{align*}
    \Vert u(t) - u_h(t) \Vert_X \leq C e^{(\lambda - \epsilon) t} t^{-\theta / r + \rho / r} h^{\theta} \Vert (\lambda + A)^{\rho / r} x \Vert_X, \quad x \in D((\lambda + A)^{\rho / r}),
\end{align*}
for $\theta \in [0,r]$ and $\rho \in [0,\theta]$. 
\end{theorem}

Before we are able to prove Theorem \ref{theorem:1}, we need to recall some properties of analytic semigroups and their generators. The analytic semigroup generated by $-A$ in \ref{cond:sectorial} is defined by the following Dunford integral,  
\begin{align}\label{eq:residue-theorem-semigroup}
    S(t) := \frac{1}{2\pi i} \int_{-\lambda + \gamma} e^{-zt} (z - A)^{-1} \, dz,
\end{align}
where $\gamma := \{ s e^{\pm i \delta}, \ s \geq R \} \cup \{ R e^{i \phi}, \ \vert \phi \vert \in [\delta, \pi] \}$, for any $R \geq 0$, oriented counterclockwise, running through $\mathbb{C} \setminus \Sigma_{\delta}$ and surrounding $\sigma(A)$ (see e.g. Theorem 3.1 in \cite{yagi}). Owing to condition \ref{cond:sectorial} the integral \eqref{eq:residue-theorem-semigroup} converges in the uniform operator topology, and is independent of the choice of $R$. With $\lambda, A$ as in \ref{cond:sectorial}, we can define fractional powers of $\lambda + A$ by, 
\begin{align}\label{eq:fractional-powers}
    (\lambda + A)^{-\alpha} := \frac{1}{\Gamma(\alpha)} \int_0^{\infty} t^{-1 + \alpha} e^{-\lambda t} S(t) \, dt, \qquad (\lambda + A)^{\alpha} := \frac{\sin(\alpha \pi)}{\pi} \int_0^{\infty} t^{-1 + \alpha} (t + \lambda + A)^{-1} (\lambda + A) \, dt, 
\end{align}
for $\alpha \in (0,1)$, while the expression for $(\lambda + A)^{-\alpha}$ also holds for any $\alpha \geq 1$ (see Equation 6.9 and Theorem 6.9 in \cite{pazy}, respectively, using that the operator semigroup generated by $-\lambda - A$, is equal to $e^{-\lambda \, \cdot} S(\cdot)$). The following properties of generators of analytic semigroups are well known.
\begin{lemma}\label{lemma:analytic-semigroup}
    Let $\lambda, A$ be as in \ref{cond:sectorial}. Then,
    \begin{enumerate}
        \item[(a)] there is $C, \epsilon > 0$ such that for any $\alpha \geq 0$, $t > 0$ and $x \in X$
        \begin{align}\label{eq:analytic-semigroup-property-1-2}
            S(t) x \in D((\lambda + A)^{\alpha}) \quad \text{with} \quad \Vert (\lambda + A)^{\alpha} S(t) x \Vert_X \leq C e^{(\lambda-\epsilon)t} t^{-\alpha} \Vert x \Vert_X,
        \end{align}
        with convention $(\lambda + A)^0 = I$, \medskip

        \item[(b)] for $\alpha, \beta \in \mathbb{R}$, $x \in D((\lambda + A)^{\gamma})$ with $\gamma = \max(\alpha, \beta, \alpha + \beta)$,
        \begin{align*}
            (\lambda + A)^{\alpha} (\lambda + A)^{\beta} x = (\lambda + A)^{\beta} (\lambda + A)^{\alpha} x = (\lambda + A)^{\alpha + \beta} x,
        \end{align*}
        
        \item[(c)] $D((\lambda + A)^{\alpha})$ is dense in $X$, \medskip
        
        \item[(d)] $(\lambda + A)^{\alpha} S(t) x = S(t) (\lambda + A)^{\alpha} x$ for $x \in D((\lambda + A)^{\alpha})$, \medskip
        
        \item[(e)] and finally for $\alpha \in [0,1]$,
        \begin{align}
        \label{eq:decay-no-lambda}
            \Vert (S(t) - I) x \Vert_X &\leq C e^{\lambda t} t^{\alpha} \Vert (\lambda + A)^{\alpha} x \Vert_X.
        \end{align}
    \end{enumerate}
\end{lemma}
\begin{proof}
    See appendix \ref{app:prelim-proofs}. It is worth noting some parameter dependence of $C$ and $\epsilon$ in \eqref{eq:analytic-semigroup-property-1-2}. We have $0 < \epsilon < \sup\{s > 0 : \sigma(\lambda - s + A) \subseteq \Sigma_{\delta'} \}$ for some $\delta' \in [\delta, \pi / 2)$, while $C$ depends on $M, \delta'$ and $\alpha$. 
\end{proof}

The next lemma asserts that the semigroup generated by $-A_h$ on $X_h$ has a smoothing effect similar to \eqref{eq:analytic-semigroup-property-1-2}. 
\begin{lemma}\label{lemma:discrete-semigroup-smoothing}
    Let $A_h$ be as in condition \ref{cond:discrete-sectorial}. Then there is $C, \epsilon > 0$, such that for any $x \in X$, $\alpha \geq 0$, 
    \begin{align*}
        \Vert (\lambda + A_h)^{\alpha} S_h(t) \pi_h x \Vert_X \leq C e^{(\lambda - \epsilon) t} t^{-\alpha} \Vert x \Vert_X. 
    \end{align*}
\end{lemma}
\begin{proof}
    See appendix \ref{app:prelim-proofs}. The parameter dependence of $C$ and $\epsilon$ are as in \eqref{eq:analytic-semigroup-property-1-2} (but with $A$ replaced by $A_h$, see e.g. Lemma \ref{lemma:shifted-discrete}). 
\end{proof}

For the proof of Theorem \ref{theorem:1}, the interpolation inequalities of Lemma \ref{lemma:interpolation-inequality} will be key. The first interpolation inequality of that lemma is a special case of the moment inequality for sectorial operators when $\phi \geq 0$ (see Section 2.7.4 \cite{yagi}), and the extension to $\phi \in [-1, 0]$ is straight forward. The second interpolation inequality is similar to the Heinz--Kato inequality (see \cite{1951-heinz}, \cite{1961-kato} and Section 2.8.3 in \cite{yagi}), but accommodates operators on Banach spaces, as opposed to only Hilbert spaces. 
\begin{lemma}\label{lemma:interpolation-inequality}
    Let $\lambda, A$ be as in condition \ref{cond:sectorial}. Then, for any $\phi \in [-1,1]$, $x \in D((\lambda + A)^{\phi}) \cap X$ and $\alpha \in [0,1]$, the following interpolation inequality holds,
    \begin{align*}
        \Vert (\lambda + A)^{\phi \alpha} x \Vert_X \leq C \Vert (\lambda + A)^{\phi} x \Vert_X^{\alpha} \Vert x \Vert_X^{1-\alpha}.
    \end{align*}
    Further, for any $B \in L(X)$ with the property that $B (\lambda + A)^{\phi} \in L(X)$, the following interpolation inequality holds,
    \begin{align*}
        \Vert B (\lambda + A)^{\phi \alpha} \Vert_{L(X)} \leq C \Vert B (\lambda + A)^{\phi} \Vert_{L(X)}^{\alpha} \Vert B \Vert_{L(X)}^{1-\alpha}. 
    \end{align*}
\end{lemma}
\begin{proof}
The proof of the first inequality for $\phi \in [0,1]$ is given in Section 2.7.4 in \cite{yagi}. The arguments are the same as those in the proof of the second inequality with $\phi \geq 0$. We include the full argument for both inequalities. 

For ease of notation, and without loss of generality, assume $\lambda = 0$. The case $\phi = 0$ is clear for any $\alpha$, and the cases $\alpha = 0$ and $\alpha = 1$ for any $\phi$. For the first inequality, the case $\phi = 1$ is given in Theorem 6.10 in \cite{pazy} (see also Proposition 2.2.15 in \cite{lunardi}). In the case of $\phi \in (0,1)$, we have $\phi \alpha \in (0,1)$, and 
\begin{align}\label{eq:setting-K-equal}
\begin{split}
    \Vert A^{\phi \alpha} x \Vert_X &= \vert \frac{\sin(\alpha \phi \pi)}{\pi} \vert \Vert \int_0^{\infty} t^{-1 + \phi \alpha} (t+A)^{-1} A x \, dt \Vert_X \\
    &\leq C \bigg(\int_0^{K} t^{-1 + \phi \alpha} \Vert (t+A)^{-1} A x \Vert_X \, dt + \int_K^{\infty} t^{-1 + \phi \alpha} \Vert (t+A)^{-1} A^{1-\phi} \Vert_{L(X)} \Vert A^{\phi} x \Vert_X \, dt \bigg) \\
    &\leq C K^{\phi \alpha} \Vert x \Vert_X + C K^{-\phi + \phi \alpha} \Vert A^{\phi} x \Vert_X. 
\end{split}
\end{align}
Here we used that since the inequality holds for $\phi = 1$, and $\Vert (t+A)^{-1} \Vert_{L(X)} \leq C t^{-1}$, $\Vert A (t+A)^{-1} \Vert_{L(X)} \leq C$ (by the bound on the resolvent in \ref{cond:sectorial} and the identity $A(t+A)^{-1} = I - t (t+A)^{-1}$) we must have
\begin{align*}
    \Vert (t + A)^{-1} A^{1-\phi} \Vert_{L(X)} = \Vert A^{1-\phi} (t+A)^{-1} \Vert_{L(X)} \leq C \Vert (t + A)^{-1} \Vert_{L(X)}^{\phi} \Vert A (t+A)^{-1} \Vert_{L(X)}^{1-\phi} \leq C t^{-\phi}.
\end{align*}
Setting $K = (\Vert A^{\phi} x \Vert_X / \Vert x \Vert_X)^{1 / \phi}$ in \eqref{eq:setting-K-equal}, we get the estimate.

For $\phi \in (-1,0)$, we have $\phi \alpha \in (-1, 0)$, $\alpha \phi - \phi \geq 0$, and by the previous inequality, 
\begin{align*}
    \Vert A^{\phi \alpha} x \Vert_X = \Vert A^{\phi (\alpha - 1)} A^{\phi} x \Vert_X \leq C \Vert A^{-\phi} A^{\phi} x \Vert_X^{1 - \alpha} \Vert A^{\phi} x \Vert_X^{\alpha} \leq C \Vert A^{\phi} x \Vert_X^{\alpha} \Vert x \Vert_X^{1-\alpha}. 
\end{align*}

The second inequality of the lemma also holds for $\alpha = 0$ and $\alpha = 1$ for any $\phi$, and any $\alpha$ when $\phi = 0$. For $\phi \in (0,1)$, arguing as in \eqref{eq:setting-K-equal}, and using the boundedness of $B$ to pass it under the integral sign, we find that, 
\begin{align*}
    \Vert B A^{\phi \alpha} \Vert_{L(X)} &= \vert \frac{\sin(\phi \alpha \pi)}{\pi} \vert \Vert \int_0^{\infty} t^{-1 + \phi \alpha} B (t + A)^{-1} A \, dt \Vert_{L(X)} \\
    &\leq C \big( \int_0^{K} t^{-1 + \phi \alpha} \Vert B \Vert_{L(X)} \Vert A (t+A)^{-1} \Vert_{L(X)} \, dt + \int_K^{\infty} t^{-1 + \phi \alpha} \Vert B A^{\phi} \Vert_{L(X)} \Vert A^{1-\phi} (t + A)^{-1} \Vert_{L(X)} \, dt \big) \\
    &\leq C K^{\phi \alpha} \Vert B \Vert_{L(X)} + C K^{-\phi + \phi \alpha} \Vert B A^{\phi} \Vert_{L(X)},
\end{align*}
Setting $K = (\Vert B A^{\phi} \Vert_{L(X)} / \Vert B \Vert_{L(X)})^{1/\phi}$ gives the inequality. 

For $\phi \in (-1,0)$, we have $\phi \alpha \in (-1, 0)$, and by using the definition of $A^{\phi}$ in \eqref{eq:fractional-powers}, the boundedness of $B$ to pass it under the integral sign, and the inequality $\Vert A^{-\phi} S(t) \Vert_{L(X)} \leq C t^{\phi}$ from Lemma \ref{lemma:analytic-semigroup}, we find
\begin{align*}
    \Vert B A^{\phi \alpha} \Vert_{L(X)} &= \frac{1}{\Gamma(\phi \alpha)} \Vert \int_0^{\infty} t^{-1 - \phi \alpha} B S(t) \, dt \Vert_{L(X)} \\
    &\leq C \bigg( \int_0^K t^{-1 - \phi \alpha} \Vert B \Vert_{L(X)} \Vert S(t) \Vert_{L(X)} \, dt + \int_K^{\infty} t^{-1 - \phi \alpha} \Vert B A^{\phi} \Vert_{L(X)} \Vert A^{-\phi} S(t) \Vert_{L(X)} \, dt \bigg) \\
    &\leq C K^{-\phi \alpha} \Vert B \Vert_{L(X)} + C K^{\phi - \phi \alpha} \Vert B A^{\phi} \Vert_{L(X)}.
\end{align*}
Setting $K = (\Vert B A^{\phi} \Vert_{L(X)} / \Vert B \Vert_{L(X)})^{-1/\phi}$, gives the inequality. 
\end{proof}

The last lemma that we list before proving Theorem \ref{theorem:1} uses the previous interpolation inequality to assert that if a collection of operators satisfy a couple of inequalities, we can interpolate these inequalities to arrive at the particular expression in Theorem \ref{theorem:1}. In the proof of Theorem \ref{theorem:1}, we will apply this lemma to the operator $e^{-\lambda t + \epsilon t}(S(t) - S_h(t) \pi_h)$, where $\epsilon > 0$ is the smallest $\epsilon$ from Lemma \ref{lemma:analytic-semigroup} and \ref{lemma:discrete-semigroup-smoothing}. 
\begin{lemma}\label{lemma:interpolation}
    Let $\lambda, A$ be as in condition \ref{cond:sectorial}, and let $B_{t,h} \in L(X)$ be a collection of bounded linear operators satisfying,
    \begin{align*}
    \Vert B_{t,h} \Vert_{L(X)} \leq C, \quad \Vert B_{t,h} (\lambda + A)^{-1} \Vert_{L(X)} \leq C h^r, \quad \Vert B_{t,h} \Vert_{L(X)} \leq C t^{-1} h^r, 
    \end{align*}
    for some $C > 0$. Then, the following estimate holds,
    \begin{align*}
        \Vert B_{t,h} x \Vert \leq C t^{-\theta / r + \rho / r} h^{\theta} \Vert (\lambda + A)^{\rho / r} x \Vert_X, \quad x \in D((\lambda + A)^{\rho / r}),
    \end{align*}
    for $\theta \in [0,r]$, and $\rho \in [0, \theta]$.
\end{lemma}
\begin{proof}
    We have,
    \begin{align*}
        \Vert B_{t,h} (\lambda + A)^{-\alpha} \Vert_{L(X)} \leq C h^{r \alpha}, \quad \Vert B_{t,h} (\lambda + A)^{-\alpha} \Vert_{L(X)} \leq C t^{-1 + \alpha} h^{r},
    \end{align*}
    for $\alpha \in \{0, 1\}$ and in view of Lemma \ref{lemma:interpolation-inequality}, we have that this holds for $\alpha \in [0,1]$. Thus, we get,
    \begin{align*}
        \Vert B_{t,h} (\lambda + A)^{-\alpha} \Vert_{L(X)} &= \Vert B_{t,h} (\lambda + A)^{-\alpha} \Vert_{L(X)}^{\beta} \Vert B_{t,h} (\lambda + A)^{-\alpha} \Vert_{L(X)}^{1- \beta} \\
        &\leq C ( t^{-1 + \alpha} h^r )^{\beta} ( h^{r \alpha} )^{1-\beta} \\
        &\leq C t^{-\theta / r + \rho / r} h^{\theta},
    \end{align*}
    where we have set $\alpha = \rho / r$ and $\beta = (\rho - \theta) / (\rho - r)$, for some $\theta \in [0,r]$, and $\rho \in [0, \theta]$.
\end{proof}

We are now ready to prove Theorem \ref{theorem:1}. 
\begin{proof}[Proof of Theorem \ref{theorem:1}]
It suffices to show the inequalities of Lemma \ref{lemma:interpolation} for the operator $e^{\epsilon t}(e^{-\lambda t} S(t) - e^{-\lambda t}S_h(t) \pi_h)$, with $\epsilon > 0$ the smallest $\epsilon$ from Lemma \ref{lemma:analytic-semigroup} and Lemma \ref{lemma:discrete-semigroup-smoothing}. Recall that the operator semigroup generated by $-\lambda - A$ is $e^{-\lambda \, \cdot} S(\cdot)$, and so without loss of generality, assume in the following that $\lambda = 0$. 

In order to show the inequalities of Lemma \ref{lemma:interpolation}, take $x \in X$, and decompose the error as follows,
\begin{align*}
    S(t) x -  S_h(t) \pi_h x &= (I -  \pi_h) S(t) x +  \pi_h S(t) x -  S_h(t) \pi_h x  \\
    &=: (I) + (II),
\end{align*}
where we have set,
\begin{align*}
    (I) = (I - \pi_h)S(t) x, \quad (II) =  \pi_h S(t) x -  S_h(t) \pi_h x.
\end{align*}
First, note that $\Vert (I) + (II) \Vert_{X} \leq C e^{-\epsilon t} \Vert x \Vert_X$, which follows by Lemma \ref{lemma:analytic-semigroup} and \ref{lemma:discrete-semigroup-smoothing}. To show the remaining inequalities we treat $(I)$ and $(II)$ separately. 

Define $\Delta_h := A^{-1} -  A_h^{-1} \pi_h$ (as in condition \ref{cond:abstract-solution-operator} with $\lambda = 0$). Next we verify that $\Vert (I) \Vert_X \leq C e^{-\epsilon t} h^r \Vert A x\Vert_X$ and $\Vert 
(I) \Vert_X \leq C e^{-\epsilon t} t^{-1} h^r \Vert x \Vert_X$. Note that,
\begin{align*}
    (I) &= (I - \pi_h) A^{-1} A S(t) x \\
    &= (A^{-1} -  \pi_h A^{-1}) S(t) A x \\
    &= (A^{-1} -  A_h^{-1} \pi_h +  A_h^{-1} \pi_h -  \pi_h A^{-1}) S(t) A x \\
    &= (\Delta_h -  \pi_h \Delta_h) S(t) A x.
\end{align*}
This in turn gives that $\Vert (I) \Vert_X \leq C \Vert \Delta_h \Vert_{L(X)} \Vert A S(t) \Vert_{L(X)} \Vert x \Vert_X$, and $\Vert (I) \Vert_X \leq C \Vert \Delta_h \Vert_{L(X)} \Vert S(t) \Vert_{L(X)} \Vert A x \Vert_X$, and by applying condition \ref{cond:abstract-solution-operator}, and Lemma \ref{lemma:analytic-semigroup}, we get the desired inequalities. 

It remains to show that $\Vert (II) \Vert_{X} \leq C e^{-\epsilon t} h^r \Vert A x \Vert_X$ and $\Vert (II) \Vert_{X} \leq C e^{-\epsilon t} t^{-1} h^r \Vert x \Vert_X$. In order to do so, we follow \cite{1998-trotter-kato}, and use an integral expression for $(II)$. To that end, define $e_h(t) := (II)$, and note that we have, 
\begin{align*}
    \dot{e}_h(t) &= - \pi_h A S(t) x +  A_h S_h(t) \pi_h x \\
    &= -  A_h ( \pi_h S(t) x -  S_h(t) \pi_h x) +  A_h \pi_h S(t) x -  \pi_h A S(t) x \\
    &= -  A_h e_h(t) + ( A_h \pi_h -  \pi_h A ) S(t) x \\
    &= -  A_h e_h(t) +  A_h \pi_h (A^{-1} -  A_h^{-1} \pi_h) A S(t) x \\
    &= -  A_h e_h(t) +  A_h \pi_h \Delta_h A S(t) x, 
\end{align*}
In the next to last equality, we used that $ \pi_h = A_h A_h^{-1} \pi_h$. Therefore (since $e_h(0) = 0$), we have, 
\begin{align*}
    e_h(t) = \int_0^t  S_h(t-s) A_h \pi_h \Delta_h A S(s) x \, ds.
\end{align*}
We split the integral above in two as in Lemma 3.1 in \cite{2013-tambue}, 
\begin{align*}
    e_h(t) = \int_0^{t/2}  S_h(t-s) A_h \pi_h \Delta_h A S(s) x \, ds + \int_{t/2}^t  S_h(t-s) A_h \pi_h \Delta_h A S(s) x \, ds =: (i) + (ii),
\end{align*}
and estimate each term separately. Integrating by parts, we get,
\begin{align*}
    (ii) =  \pi_h \Delta_h A S(t) x -  S_h(t/2) \pi_h \Delta_h A S(t/2) x + \int_{t/2}^t  S_h(t-s) \pi_h \Delta_h A^2 S(s) x \, ds, 
\end{align*}
where we used the identity,
\begin{align*}
    &\pi_h \Delta_h A S(t) x -  S_h(t/2) \pi_h \Delta_h A S(t/2) x \\
    &\qquad= \int_{t/2}^t \frac{d}{ds} \bigg(  S_h(t-s) \pi_h \Delta_h A S(s) x \bigg) \, ds \\
    &\qquad= \int_{t/2}^t  S_h(t-s) A_h \pi_h \Delta_h A S(s) x \, ds - \int_{t/2}^t  S_h(t-s) \pi_h \Delta_h A^2 S(s) x \, ds.
\end{align*}
We start by showing $\Vert (II) \Vert_X \leq C e^{-\epsilon t} h^r \Vert A x \Vert_X$. For $(i)$, we have by condition \ref{cond:abstract-solution-operator}, Lemma \ref{lemma:analytic-semigroup} and \ref{lemma:discrete-semigroup-smoothing},
\begin{align*}
    \Vert \int_0^{t/2} S_h(t-s) A_h \pi_h \Delta_h S(s) A x \, ds \Vert_X &\leq \int_0^{t/2} \Vert A_h S_h(t-s) \pi_h \Vert_{L(X)} \Vert \Delta_h \Vert_{L(X)} \Vert S(s) \Vert_{L(X)} \Vert A x \Vert_X \, ds \\
    &\leq C e^{-\epsilon t} \int_0^{t/2} (t-s)^{-1} h^r \Vert A x \Vert_X \, ds \\
    &\leq C e^{-\epsilon t} h^r \Vert A x \Vert_X,
\end{align*}
For the first and second term if $(ii)$, we get using the boundedness of the semigroups and of the projection $\pi_h$, in addition to condition \ref{cond:abstract-solution-operator}:
\begin{align*}
    \Vert \pi_h \Delta_h S(t) A x \Vert_X &\leq C \Vert \Delta_h \Vert_{L(X)} \Vert S(t) \Vert_{L(X)} \Vert A x \Vert_X \leq C e^{-\epsilon t} h^r \Vert A x \Vert_X, \\
    \Vert S_h(t/2) \pi_h \Delta_h S(t/2) A x \Vert_X &\leq \Vert S_h(t/2) \pi_h \Vert_{L(X)} \Vert \Delta_h \Vert_{L(X)} \Vert S(t/2) \Vert_{L(X)} \Vert A x \Vert_X \leq C e^{-\epsilon t} h^r \Vert A x \Vert_X. 
\end{align*}
For the last term in $(ii)$, we have by Lemma \ref{lemma:analytic-semigroup} and \ref{lemma:discrete-semigroup-smoothing}, condition \ref{cond:abstract-solution-operator}
\begin{align*}
    \Vert \int_{t/2}^t S_h(t-s) \pi_h \Delta_h A^2 S(s) x \, ds \Vert_X &\leq \int_{t/2}^t \Vert S_h(t-s) \pi_h \Vert_{L(X)} \Vert \Delta_h \Vert_{L(X)} \Vert A S(s) \Vert_{L(X)} \Vert Ax \Vert_X \, ds \\
    &\leq C e^{-\epsilon t} \int_{t/2}^t s^{-1} h^r \Vert A x \Vert_X \, ds \\
    &\leq C e^{-\epsilon t} h^r \Vert A x \Vert_X.
\end{align*}
We therefore get that $\Vert (II) \Vert \leq C e^{-\epsilon t} h^r \Vert A x \Vert_X$. 

It remains to show $\Vert (II) \Vert_X \leq C e^{-\epsilon t} t^{-1} h^r \Vert x \Vert_X$. Using integration by parts on $(i)$ as above, we get,
\begin{align*}
    (i) = -S_h(t/2) A_h \pi_h \Delta_h S(t/2) +  S_h(t) A_h \pi_h \Delta_h + \int_0^{t/2} A_h^2 S_h(t-s) \pi_h \Delta_h S(s) x \, ds.
\end{align*} 
We note that,
\begin{align*}
    \Vert A_h S_h(t/2) \pi_h \Delta_h S(t/2) x \Vert_{X} \leq \Vert A_h S_h(t/2) \pi_h \Vert_{L(X)} \Vert \Delta_h \Vert_{L(X)} \Vert S(t/2) x \Vert_X \leq C e^{-\epsilon t} t^{-1} h^r \Vert x \Vert_X,
\end{align*}
and,
\begin{align*}
    \Vert S_h(t) A_h \pi_h \Delta_h x \Vert_X \leq \Vert A_h S_h(t) \pi_h \Vert_{L(X)} \Vert \Delta_h \Vert_{L(X)} \Vert x \Vert_X \leq C e^{-\epsilon t} t^{-1} h^r \Vert x \Vert_X, 
\end{align*}
while for the third term of $(i)$, we have,
\begin{align*}
    \Vert \int_0^{t/2} S_h(t-s) A_h^2 \pi_h \Delta_h S(s) x \, ds \Vert_{X} &\leq \int_0^{t/2} \Vert A_h^2 S_h(t-s) \pi_h \Vert_{L(X)} \Vert \Delta_h \Vert_{L(X)} \Vert S(s) \Vert_{L(X)} \Vert x \Vert_X \, ds \\
    &\leq C e^{-\epsilon t} \int_0^{t/2} (t-s)^{-2} h^r \Vert x \Vert_X \, ds \\
    &\leq C e^{-\epsilon t} t^{-1} h^r \Vert x \Vert_X. 
\end{align*}
For the first two terms of $(ii)$ (after integration by parts) we have using condition \ref{cond:abstract-solution-operator}, Lemma \ref{lemma:analytic-semigroup} and \ref{lemma:discrete-semigroup-smoothing}
\begin{align*}
    \Vert \pi_h \Delta_h A S(t) x \Vert_X \leq C \Vert \Delta_h \Vert_{L(X)} \Vert A S(t) \Vert_{L(X)} \Vert x \Vert_X \leq C e^{-\epsilon t} t^{-1} h^r \Vert x \Vert_X,
\end{align*}
and
\begin{align*}
    \Vert S_h(t/2) \pi_h \Delta_h A S(t / 2) x \Vert_X \leq \Vert S_h(t/2) \pi_h \Vert_{L(X)} \Vert \Delta_h \Vert_{L(X)} \Vert A S(t/2) \Vert_{L(X)} \Vert x \Vert_X \leq C e^{-\epsilon t} t^{-1} h^r \Vert x \Vert_X.
\end{align*}
By the same arguments, we have for the last term of $(ii)$, 
\begin{align*}
    \Vert \int_{t/2}^t S_h(t-s) \pi_h \Delta_h A^2 S(s) x \, ds \Vert_X &\leq \int_{t/2}^t \Vert S_h(t-s) \pi_h \Vert_{L(X)} \Vert \Delta_h \Vert_{L(X)} \Vert A^2 S(s) \Vert_{L(X)} \Vert x \Vert_X \, ds \\
    &\leq C e^{-\epsilon t} \int_0^{t/2} s^{-2} h^r \Vert x \Vert_X \, ds \\
    &\leq C e^{-\epsilon t} t^{-1} h^r \Vert x \Vert_X.
\end{align*}
Therefore, we get that $\Vert (II) \Vert_X \leq C e^{-\epsilon t} t^{-1} h^r \Vert x \Vert_X$.

Since all inequalities of Lemma \ref{lemma:interpolation} have been shown for $e^{\epsilon t}(S(t) - S_h(t) \pi_h)$, the proof is finished. 
\end{proof}

\subsection{Fully discrete error estimate}

By combining $A_h$ in Theorem \ref{theorem:1} with a rational approximation of the exponential function, we get a fully discrete approximation of $S(t)$. To that end, we let $r(z) := (1 + z)^{-1}$, partition the time interval $[0,T]$ uniformly with increments $\Delta t = T / N$ for some positive integer $N$, and define, 
\begin{align}\label{eq:fully-discrete-semigroup-operator}
    S_{h,\Delta t}(t) :=
    \begin{cases}
        I, \quad &t = 0, \\
        \sum_{n = 0}^{N-1} r(\Delta t A_{h} )^{n+1} \chi_{(t_n,t_{n+1}]}(t), \quad &t > 0.
    \end{cases}
\end{align}
Here $t_n = n \Delta t$, $I$ is the identity, and $\chi_D(\cdot)$ is the characteristic function on the set $D$. 

In this section we combine Theorem \ref{theorem:1} with Lemma \ref{lemma:fully-discrete-inequalities-1}, \ref{lemma:fully-discrete-inequalities-2} and \ref{lemma:fully-discrete-inequalities-3} in order to arrive at an error estimate for our fully discrete approximation, which is stated in Theorem \ref{theorem:fully-discrete-semigroup-approximation-2}. The proofs of Lemma \ref{lemma:fully-discrete-inequalities-1}, \ref{lemma:fully-discrete-inequalities-2} and \ref{lemma:fully-discrete-inequalities-3} share some similarity to those of Theorem 9.1, 9.2 and 9.3 in \cite{thomee}, with the difference being that those theorems assume that we can always choose $\lambda = 0$---we need to accommodate the cases where $\lambda$ has to be chosen greater than $0$. This is e.g. the case for the example we consider in Section 3. To accommodate these cases, we are going to pay with a factor that grows exponentially with time. When we may choose $\lambda = 0$ we recover the results from Chapter 9 in \cite{thomee}, with an additional factor decaying exponentially with time. 

Lemma \ref{lemma:fully-discrete-inequalities-4} and \ref{lemma:fully-discrete-inequalities-5} are similar to Lemma \ref{lemma:fully-discrete-inequalities-1}, \ref{lemma:fully-discrete-inequalities-2} and \ref{lemma:fully-discrete-inequalities-3}, and are needed to get a fully discrete error estimate for less regular data, stated in Theorem \ref{theorem:fully-discrete-semigroup-approximation-4}. The proofs of the five following lemmas are all based on a Dunford integral representation of the difference,
\begin{align*}
    S_h(t_n) - S_{h,\Delta t}(t_n) = \frac{1}{2\pi i} \int_{-\lambda + \gamma} (e^{-z t_n} - r(\Delta t z)^n) (z - A_h)^{-1} \, dz,
\end{align*}
where $\gamma$ is as in \eqref{eq:residue-theorem-semigroup}. The details are given in Appendix \ref{app:rational-approximation-proofs}. 
\begin{lemma}\label{lemma:fully-discrete-inequalities-1}
Suppose condition \ref{cond:discrete-sectorial} holds, and let $S_{h,\Delta t}$ be as in \eqref{eq:fully-discrete-semigroup-operator}. Then there is $C, c, \epsilon > 0$ such that
\begin{align*}
    \Vert (S_h(t_n) - S_{h,\Delta t}(t_n)) \pi_h x \Vert_X \leq C e^{c (\lambda-\epsilon) t_n} \Vert x \Vert_X.
\end{align*}
\end{lemma}
\begin{lemma}\label{lemma:fully-discrete-inequalities-2}
Under the same conditions as in Lemma \ref{lemma:fully-discrete-inequalities-1},
\begin{align*}
    \Vert (S_h(t_n) - S_{h,\Delta t}(t_n)) \pi_h x \Vert_X \leq C e^{c (\lambda-\epsilon) t_n} \Delta t \Vert (\lambda+A_h) \pi_h x \Vert_X.
\end{align*}
\end{lemma}
\begin{lemma}\label{lemma:fully-discrete-inequalities-3}
Under the same conditions as in Lemma \ref{lemma:fully-discrete-inequalities-1},
\begin{align*}
    \Vert (S_h(t_n) - S_{h,\Delta t}(t_n)) \pi_h x \Vert_X \leq C e^{c (\lambda-\epsilon) t_n} t_n^{-1} \Delta t \Vert x \Vert_X.
\end{align*}
\end{lemma}
\begin{lemma}\label{lemma:fully-discrete-inequalities-4}
Under the same conditions as in Lemma \ref{lemma:fully-discrete-inequalities-1},
\begin{align*}
    \Vert (S_h(t_n) - S_{h,\Delta t}(t_n)) \pi_h x \Vert_X \leq C e^{c(\lambda-\epsilon) t_n} t_n^{-1/2} \Vert (\lambda + A_h)^{-1/2} \pi_h x \Vert_X. 
\end{align*}
\end{lemma}
\begin{lemma}\label{lemma:fully-discrete-inequalities-5}
Under the same conditions as in Lemma \ref{lemma:fully-discrete-inequalities-1},
\begin{align*}
    \Vert (S_h(t_n) - S_{h,\Delta t}(t_n)) \pi_h x \Vert_X \leq C e^{c(\lambda-\epsilon) t_n} t_n^{-1} \Delta t^{1/2} \Vert (\lambda + A_h)^{-1/2} \pi_h x \Vert_X.
\end{align*}
\end{lemma}

The following condition and lemma will be useful for ``replacing $A_h$ by $A$" in the following error estimates.
\begin{description}
    \myitem{(A6)}\label{cond:discrete-growth} For some $C > 0$, the following estimate holds, 
    \begin{align}\label{eq:discrete-growth}
        \Vert A_h x \Vert_X \leq C h^{-r} \Vert x \Vert_X, \quad \text{for any } x \in X_h.
    \end{align}
\end{description}
\begin{lemma}\label{lemma:discrete-operator-bound}
    Under the conditions of Assumuption \ref{assumption:abstract-model} and condition \ref{cond:discrete-growth}, there is some $C > 0$, such that,
    \begin{align*}
        \Vert (\lambda + A_h) \pi_h x \Vert_X \leq C \Vert (\lambda + A) x \Vert_X, \quad \text{for any } x \in X.
    \end{align*}
\end{lemma}
\begin{proof} 
    Due to \ref{cond:discrete-growth}, we have $\Vert (\lambda + A_h) x \Vert_{X} \leq C h^{-r} \Vert x \Vert_X$ as well. It suffices to show that $\Vert (\lambda + A_h) \pi_h (\lambda + A)^{-1} x \Vert_X \leq C \Vert x \Vert_X$, for any $x \in X$. To that end,  
    \begin{align*}
        \Vert (\lambda + A_h) \pi_h (\lambda + A)^{-1} x \Vert_X &= \Vert (\lambda + A_h) \pi_h ((\lambda+A)^{-1} - (\lambda+A_h)^{-1} \pi_h + (\lambda+A_h)^{-1} \pi_h) x \Vert \\
        &\leq \Vert \pi_h x \Vert_X + \Vert (\lambda + A_h) \pi_h \Vert_{L(X_h)} \Vert ((\lambda+A)^{-1} - (\lambda+A_h)^{-1} \pi_h) x \Vert_X \\
        &\leq C \Vert x \Vert_X,
    \end{align*}
    where we used condition \ref{cond:abstract-solution-operator}.  
\end{proof}

We will need the following lemma to show our first fully discrete error estimate. 
\begin{lemma}\label{lemma:fully-discrete-semigroup-approximation-1}
    Suppose the conditions of Assumption \ref{assumption:abstract-model} and \ref{cond:discrete-growth} holds. Then, there is $C, c, \epsilon > 0$ such that
    \begin{align*}
        \Vert (S_h(t_n) - S_{h,\Delta t}(t_n) ) \pi_h x \Vert_X \leq C e^{c (\lambda-\epsilon) t_n} t_n^{-\theta/r + \rho/r} \Delta t^{\theta/r} \Vert (\lambda + A)^{\rho/r} x \Vert_X, 
    \end{align*}
    for $\theta \in [0,r]$ and $\rho \in [0, \theta]$.
\end{lemma}
\begin{proof}
Set $B_{t_n, \Delta t} := e^{-c(\lambda-\epsilon) t_n}(S_h(t_n) - S_{h,\Delta t}(t_n)) \pi_h$, where $c$ is either the biggest or smallest $c$ from Lemma \ref{lemma:fully-discrete-inequalities-1}, \ref{lemma:fully-discrete-inequalities-2} and \ref{lemma:fully-discrete-inequalities-3}, depending on the sign of $\lambda - \epsilon$. By combining the inequalities of Lemma \ref{lemma:fully-discrete-inequalities-1}, \ref{lemma:fully-discrete-inequalities-2} and \ref{lemma:fully-discrete-inequalities-3}, with Lemma \ref{lemma:discrete-operator-bound}, we have,
\begin{align*}
    &\Vert B_{t_n,\Delta t} \Vert_{L(X)} \leq C, \quad \Vert B_{t_n,\Delta t} (\lambda + A)^{-1} \Vert_{L(X)} \leq C \Delta t, \quad \Vert B_{t_n,\Delta t} \Vert_{L(X)} \leq C t_n^{-1} \Delta t,
\end{align*}
Therefore, we have,
\begin{align*}
    \Vert B_{t_n,\Delta t} (\lambda+A)^{-\alpha} \Vert_{L(X)} \leq C \Delta t^{\alpha}, \quad \Vert B_{t_n,\Delta t} (\lambda+A)^{-\alpha} \Vert_{L(X)} \leq C t_n^{-1 + \alpha} \Delta t,
\end{align*}
for $\alpha \in \{ 0,1\}$ and therefore in $[0,1]$ by Lemma \ref{lemma:interpolation-inequality}. Finally, this gives,
\begin{align*}
    \Vert B_{t_n,\Delta t} (\lambda+A)^{-\alpha} \Vert_{L(X)} &= \Vert B_{t_n,\Delta t} (\lambda+A)^{-\alpha} \Vert_{L(X)}^{\beta} \Vert B_{t_n,\Delta t} (\lambda+A)^{-\alpha} \Vert_{L(X)}^{1-\beta} \\
    &\leq C ( t_n^{\alpha - 1} \Delta t )^{\beta} ( \Delta t^{\alpha} )^{1-\beta} \\
    &= C \Delta t^{\theta / r} t_n^{-\theta / r + \rho / r},
\end{align*}
where we have set $\alpha = \rho / r$, $\theta = \rho + (r-\rho)\beta \in [\rho,r]$. Thus $\theta \in [0,r]$ and $\rho \in [0, \theta]$. This finishes the proof.
\end{proof}

We are now ready to derive our first fully discrete error estimate.
\begin{theorem}\label{theorem:fully-discrete-semigroup-approximation-2}
    Suppose the conditions of Assumption \ref{assumption:abstract-model} and \ref{cond:discrete-growth} holds. Then, there is $C, c, \epsilon > 0$ such that
    \begin{align*}
        \Vert (S(t) - S_{h,\Delta t}(t) \pi_h) x \Vert_X \leq C e^{c (\lambda-\epsilon) t} t^{-\theta/r + \rho/r} (h^{\theta} + \Delta t^{\theta / r}) \Vert (\lambda+A)^{\rho/2} x \Vert_X, 
    \end{align*}
    for $\theta \in [0,r]$ and $\rho \in [0, \theta]$. 
\end{theorem}
\begin{proof}
For $t \in (t_n, t_{n+1}]$, we decompose the error as follows, 
\begin{align*}
    &S(t) - S(t_{n+1}) \\
    &\qquad +S(t_{n+1}) -  S_h(t_{n+1}) \pi_h \\
    &\qquad + S_h(t_{n+1}) \pi_h -  S_{h,\Delta t} (t_{n+1}) \pi_h \\
    &\qquad + S_{h,\Delta t} (t_{n+1}) \pi_h -  S_{h,\Delta t} (t) \pi_h.
\end{align*}
For the first term, we have for some $\theta \in [0,r], \rho \in [0,\theta]$,
\begin{align*}
    \Vert (S(t) - S(t_{n+1})) x \Vert_X &= \Vert (I - S(t_{n+1} - t) ) S(t) x \Vert_X \\
    &= \Vert (I - S(t_{n+1} - t)) (\lambda + A)^{-\theta / r} (\lambda + A)^{\theta / r - \rho / r}  S(t) (\lambda + A)^{\rho / r} x \Vert_X \\
    &\leq \Vert (I - S(t_{n+1} - t)) (\lambda + A)^{-\theta / r} \Vert_{L(X)} \Vert (\lambda + A)^{\theta / r - \rho / r}  S(t) \Vert_{L(X)} \Vert (\lambda + A)^{\rho / r} x \Vert_X.
\end{align*}
Using Lemma \ref{lemma:analytic-semigroup} we have,
\begin{align*}
    \Vert (I - S(t_{n+1} - t)) (\lambda + A)^{-\theta / r} \Vert_{L(X)} \leq C \Delta t^{\theta / r} \quad \text{and} \quad \Vert (\lambda + A)^{\theta / r - \rho / r} S(t) \Vert_{L(X)} \leq C e^{(\lambda-\epsilon) t} t^{-\theta / r + \rho / r},
\end{align*}
for $\theta \in [0,r]$ and $\rho \in [0, \theta]$. This finishes the first term. 

For the two next terms, we use Theorem \ref{theorem:1} and Lemma \ref{lemma:fully-discrete-semigroup-approximation-1}. For the last term, we note that,
\begin{align*}
    S_{h,\Delta t} (t_{n+1}) \pi_h - S_{h,\Delta t} (t) \pi_h = 0, \quad t \in (t_n,t_{n+1}].
\end{align*}
To replace $t_{n+1}^{-\theta / r + \rho / r}$ with $t^{-\theta / r + \rho / r}$ we use that since $-\theta / r + \rho / r \leq 0$, we have that $t_{n+1}^{-\theta / r + \rho / r} \leq t^{-\theta / r + \rho / r}$. 
\end{proof}

\subsection{Semidiscrete and fully discrete error estimates in different norms, and for less regular data}

The rest of this section is devoted to showing semidiscrete error estimates in the $\Vert (\lambda + A)^{\alpha} \cdot \Vert_X$-norm, for $\alpha \in [-1/2,1/2]$, in addition to semidiscrete and fully discrete error estimates of the form of Theorem \ref{theorem:1} and Theorem \ref{theorem:fully-discrete-semigroup-approximation-2} where $\rho$ can be chosen negative. In order to do so, we require stronger assumptions than those we managed with in the previous two subsections. The first additional condition that we need is similar to \ref{cond:abstract-solution-operator}, with the error computed in a different norm. 
\begin{description}
    \myitem{(A7)}\label{cond:abstract-solution-operator-2} For $\lambda, A$ as in \ref{cond:sectorial}, the following estimate also holds,
    \begin{align*}
        \Vert (\lambda + A)^{1/2} ((\lambda + A)^{-1} - (\lambda + A_h)^{-1} \pi_h) x \Vert_X \leq C h^{r/2} \Vert x \Vert_X, \quad x \in X.
    \end{align*}
\end{description}
With this additional condition, we are able to show another semidiscrete error estimate, which is stated in the following lemma. 
\begin{lemma}\label{lemma:1}
Suppose the conditions of Assumption \ref{assumption:abstract-model}, \ref{cond:discrete-growth} and \ref{cond:abstract-solution-operator-2} holds. Then there is $C, \epsilon > 0$ such that 
\begin{align*}
    \Vert (\lambda + A)^{\alpha} (u(t) - u_h(t)) \Vert_X \leq C e^{(\lambda - \epsilon) t} t^{-\theta / r} h^{\theta - \alpha r} \Vert x \Vert_X, 
\end{align*}
for $\alpha \in [0,1/2]$ and $\theta \in [r\alpha,r]$.
\end{lemma}
For the proof of Lemma \ref{lemma:1}, we need the following couple of lemmas. 
\begin{lemma}\label{lemma:interpolation-2}
    Let $B_{t,h} \in L(X)$ be a collection of bounded linear operators satisfying,
    \begin{align*}
        \Vert B_{t,h} \Vert_{L(X)} \leq C t^{-1/2}, \quad \Vert B_{t,h} \Vert_{L(X)} \leq C t^{-1} h^{r/2}, 
    \end{align*}
    for some $C > 0$. Then, the following estimate holds,
    \begin{align*}
        \Vert B_{t,h} x \Vert \leq C t^{-\theta / r} h^{\theta - r/2} \Vert x \Vert_X, \quad x \in X, 
    \end{align*}
    for $\theta \in [r/2,r]$. 
\end{lemma}
\begin{proof}
    We have for $\beta \in [0,1]$, 
    \begin{align*}
        \Vert B_{t,h} \Vert_{L(X)} &= \Vert B_{t,h} \Vert_{L(X)}^{\beta} \Vert B_{t,h} \Vert_{L(X)}^{1- \beta} \\
        &\leq C ( t^{-1/2} )^{\beta} ( t^{-1} h^{r/2} )^{1-\beta} \\
        &\leq C t^{-\theta / r } h^{\theta - r/2},
    \end{align*}
    where we have set and $\theta = r - r \beta / 2 \in [r/2,r]$. 
\end{proof}
\begin{lemma}\label{lemma:interpolation-new-norm}
    Let $B_{t,h} \in L(X)$ be a collection of bounded linear operators satisfying the inequalities of Lemma \ref{lemma:interpolation}, and assume that $(\lambda + A)^{1/2} B_{t,h}$ satisfies the inequalities of Lemma \ref{lemma:interpolation-2}. Then, 
    \begin{align*}
        \Vert (\lambda + A)^{\alpha} B_{t,h} x \Vert_X \leq C t^{-\theta / r} h^{\theta - \alpha r} \Vert x \Vert_X
    \end{align*}
    where $\alpha \in [0,1/2]$ and $\theta \in [r\alpha, r]$. 
\end{lemma}
\begin{proof}
By using Lemma \ref{lemma:interpolation-inequality} on the inequalities from Lemma \ref{lemma:interpolation} with $\rho = 0$ and Lemma \ref{lemma:interpolation-2}, we get for $\alpha \in [0,1/2]$, and $\beta \in \{ 0, 1 \}$,
\begin{align*}
    \Vert (\lambda + A)^{\alpha} B_{t,h} \Vert_{L(X)} &\leq C \Vert (\lambda + A)^{1/2} B_{t,h} \Vert_{L(X)}^{2 \alpha} \Vert B_{t,h} \Vert_{L(X)}^{1 - 2\alpha} \\
    &\leq C ( t^{-\theta_1 / r} h^{\theta_1 - r / 2})^{2 \alpha} ( t^{-\theta_2 / r} h^{\theta_2})^{1 - 2\alpha} \\
    &\leq C t^{-\theta / r } h^{\theta - \alpha r}, 
\end{align*}
where $\theta_1 \in [r/2,r], \theta_2 \in [0, r]$, $\theta = 2 \alpha \theta_1 + (1 - 2\alpha) \theta_2 \in [\alpha r, r]$. 
\end{proof}

\begin{proof}[Proof of Lemma \ref{lemma:1}]
The proof is similar in idea to the proof of Theorem \ref{theorem:1}. We need to show the inequalities of Lemma \ref{lemma:interpolation-2} for the operator $e^{-\lambda t + \epsilon t} (\lambda + A)^{1/2} (S(t) - S_h(t) \pi_h)$, and then employ Lemma \ref{lemma:interpolation-new-norm} to get the desired inequality (since we have the remaining inequalities for the operator $e^{-\lambda t + \epsilon t}(S(t) - S_h(t) \pi_h)$ by the proof of Theorem \ref{theorem:1}). Here, $\epsilon$ is the same as that in Theorem \ref{theorem:1}. We can, as in the proof of Theorem \ref{theorem:1}, assume without loss of generality and for ease of notation that $\lambda = 0$, since $e^{-\lambda \, \cdot} S(\cdot)$ is the operator semigroup generated by $-\lambda - A$. Therefore, it suffices to show that 
\begin{align*}
    \Vert A^{1/2} (S(t) - S_h(t) \pi_h) x \Vert_X &\leq C e^{-\epsilon t} t^{-1/2} \Vert x \Vert_X, \\
    \Vert A^{1/2} (S(t) - S_h(t) \pi_h) x \Vert_X &\leq C e^{-\epsilon t} t^{-1} h^{r/2} \Vert x \Vert_X. 
\end{align*}

Similarly as in the proof of Theorem \ref{theorem:1} we let $\Delta_h := A^{-1} - A_h^{-1} \pi_h$. Let now $e_h(t) := S(t) x - S_h(t) \pi_h x$. Then,
\begin{align*}
    \dot{e}_h(t) &= -A e_h(t) + (A_h - A) S_h(t) \pi_h x \\
    &= -A e_h(t) + A \Delta_h A_h S_h(t) \pi_h x.
\end{align*}
It may be seen that condition \ref{cond:abstract-solution-operator-2} implies $X_h \subseteq D(A^{1/2})$, and since we are applying $A$ to members of $X_h$, the equality above is understood in a $D(A^{-1/2})$-sense, which we understand as the completion of $X$ using the norm $\Vert A^{-1/2} \cdot \Vert_X$. Using that $A^{1/2}$ is closed to pass it under the integral sign, we get (note that $e_h(0) = (I - \pi_h) x$), 
\begin{align*}
    A^{1/2} e_h(t) = A^{1/2} S(t) (I - \pi_h) x + \int_0^t A S(t-s) A^{1/2} \Delta_h A_h S_h(s) \pi_h x \, ds = (I) + (II). 
\end{align*}

For $(I)$ we have:
\begin{align*}
    \Vert (I) \Vert_X &= \Vert A S(t) A^{1/2} (A^{-1} - A^{-1} \pi_h) x \Vert_X \\
    &= \Vert A S(t) A^{1/2} (\Delta_h - \Delta_h \pi_h) x \Vert_X \\
    &\leq \Vert A S(t) \Vert_{L(X)} \Vert \Delta_h (I - \pi_h) x \Vert_X,
\end{align*}
so that $\Vert (I) \Vert_X \leq C e^{-\epsilon t} t^{-1} h^{r/2} \Vert x \Vert_X$, by condition \ref{cond:abstract-solution-operator-2}, the boundedness of $I - \pi_h$ and Lemma \ref{lemma:analytic-semigroup}. Finally, $\Vert (I) \Vert_X \leq C e^{-\epsilon t} t^{-1/2} \Vert x \Vert_X$ by Lemma \ref{lemma:analytic-semigroup} and the boundedness of $I - \pi_h$.

For the second term, we have,
\begin{align*}
    (II) = \int_0^{t/2} A S(t-s) A^{1/2} \Delta_h A_h S_h(s) \pi_h x \, ds + \int_{t/2}^t A S(t-s) A^{1/2} \Delta_h A_h S_h(s) \pi_h x \, ds = (i) + (ii),
\end{align*}
and using integration by parts, we have,
\begin{align*}
    (i) = -A S(t/2) A^{1/2} \Delta_h S_h(t/2) \pi_h x + A S(t) A^{1/2} \Delta_h \pi_h x + \int_0^{t/2} A^2 S(t-s) A^{1/2} \Delta_h S_h(s) \pi_h x \, ds,
\end{align*}
where we used that,
\begin{align*}
    \int_0^{t/2} \frac{d}{ds} \bigg(  A S(t-s) A^{1/2} \Delta_h S_h(s) \pi_h x \bigg) \, ds &= \int_0^{t/2} A^2 S(t-s) A^{1/2} \Delta_h S_h(s) \pi_h x \, ds \\
    &\quad - \int_0^{t/2} A S(t-s) A^{1/2} \Delta_h A_h S_h(s) \pi_h x \, ds \\
    &= A S(t/2) A^{1/2} \Delta_h S_h(t/2) \pi_h x - A S(t) A^{1/2}\Delta_h \pi_h x.
\end{align*}
Using integrating by parts on $(ii)$ as well, gives, 
\begin{align*}
    (ii) = A^{1/2} \Delta_h A_h S_h(t) \pi_h x - S(t/2) A^{1/2} \Delta_h A_h S_h(t/2) \pi_h x + \int_{t/2}^t S(t-s) A^{1/2} \Delta_h A_h^2 S_h(s) \pi_h x \, ds, 
\end{align*}
where we used that,
\begin{align*}
    \int_{t/2}^t \frac{d}{ds} \bigg( S(t-s) A^{1/2} \Delta_h A_h S_h(s) \pi_h x \bigg) \, ds &= \int_{t/2}^t A S(t-s) A^{1/2} \Delta_h A_h S_h(s) \pi_h x \, ds \\
    &\quad - \int_{t/2}^t S(t-s) A^{1/2} \Delta_h A_h^2 S_h(s) \pi_h x \, ds \\
    &= A^{1/2} \Delta_h A_h S_h(t) \pi_h x - S(t/2) A^{1/2} \Delta_h A_h S_h(t/2) \pi_h x.
\end{align*}
We start by showing $\Vert (II) \Vert_X \leq C e^{-\epsilon t} t^{-1} h^{r/2} \Vert x \Vert_X$. For the first two terms of $(i)$, we have by Lemma \ref{lemma:analytic-semigroup}, condition \ref{cond:abstract-solution-operator-2} and Lemma \ref{lemma:discrete-semigroup-smoothing},
\begin{align*}
    \Vert A S(t/2) A^{1/2} \Delta_h S_h(t/2) \pi_h x \Vert_X &\leq \Vert A S(t/2) \Vert_{L(X)} \Vert A^{1/2} \Delta_h \Vert_{L(X)} \Vert S_h(t/2) \pi_h \Vert_{L(X)} \Vert x \Vert_X \leq C e^{-\epsilon t} t^{-1} h^{r/2} \Vert x \Vert_X, \\
    \Vert A S(t) A^{1/2} \Delta_h \pi_h x \Vert_X &\leq C \Vert A S(t) \Vert_{L(X)} \Vert A^{1/2} \Delta_h \Vert_{L(X)} \Vert x \Vert_X \leq C e^{-\epsilon t} t^{-1} h^{r/2} \Vert x \Vert_X,
\end{align*}
while for the third term, we have,
\begin{align*}
    \Vert \int_0^{t/2} A^2 S(t-s) A^{1/2} \Delta_h S_h(s) \pi_h x \, ds \Vert_X &\leq \int_0^{t/2} \Vert A^2 S(t-s) \Vert_{L(X)} \Vert A^{1/2} \Delta_h \Vert_{L(X)} \Vert 
    S_h(s) \pi_h \Vert_{L(X)} \Vert x \Vert_X \, ds \\
    &\leq C e^{-\epsilon t} \int_0^{t/2} (t-s)^{-2} h^{r/2} \Vert x \Vert_X \, ds \\
    &\leq C e^{-\epsilon t} t^{-1} h^{r/2} \Vert x \Vert_X.
\end{align*}
For the three terms of $(ii)$, we get similarly,
\begin{align*}
    \Vert A^{1/2} \Delta_h A_h S_h(t) \pi_h x \Vert_X &\leq C e^{-\epsilon t} t^{-1} h^{r/2} \Vert x \Vert_X, \\
    \Vert S(t/2) A^{1/2} \Delta_h A_h S_h(t/2) \pi_h x \Vert_X &\leq C e^{-\epsilon t} t^{-1} h^{r/2} \Vert x \Vert_X,
\end{align*}
and, 
\begin{align*}
    \Vert \int_{t/2}^t S(t-s) A^{1/2} \Delta_h A_h^2 S_h(s) \pi_h x \, ds \Vert_X &\leq \int_{t/2}^t \Vert S(t-s) \Vert_{L(X)} \Vert A^{1/2} \Delta_h \Vert_{L(X)} \Vert A_h^2 S_h(s) \pi_h \Vert_{L(X)} \Vert x \Vert_X \, ds \\
    &\leq C e^{-\epsilon t} \int_{t/2}^t s^{-2} h^{r/2} \Vert x \Vert_X \, ds \\
    &\leq C e^{-\epsilon t} t^{-1} h^{r/2} \Vert x \Vert_X. 
\end{align*}

It only remains to show that $\Vert (II) \Vert_X \leq C e^{-\epsilon t} t^{-1/2} \Vert x \Vert_X$. For $(i)$ we have,
\begin{align*}
    \Vert \int_0^{t/2} A S(t-s) A^{1/2} \Delta_h A_h S_h(s) \pi_h x \, ds \Vert_X &\leq \int_0^{t/2} \Vert A^{3/2} S(t-s) \Vert_{L(X)} \Vert \Delta_h  A_h \pi_h \Vert_{L(X)} \Vert S_h(s) \pi_h \Vert_{L(X)} \Vert x \Vert_X \, ds \\
    &\leq C e^{-\epsilon t} \int_0^{t/2} (t-s)^{-3/2} \, ds \, \Vert x \Vert_X \\
    &\leq C e^{-\epsilon t} t^{-1/2} \Vert x \Vert_X,
\end{align*}
where we used that $\Vert \Delta_h A_h \pi_h \Vert_{L(X)} \leq C$ due to condition \ref{cond:abstract-solution-operator} and \ref{cond:discrete-growth}. For the first term of $(ii)$, we have
\begin{align*}
    \Vert A^{1/2} \Delta_h A_h S_h(t) \pi_h x \Vert_X &\leq \Vert A^{1/2} \Delta_h \Vert_{L(X)} \Vert A_h^{1/2} \pi_h \Vert_{L(X)} \Vert A_h^{1/2} S_h(t) \pi_h \Vert_{L(X)} \Vert x \Vert_X \\
    &\leq C e^{-\epsilon t} h^{r/2} h^{-r/2} t^{-1/2} \Vert x\Vert_X, 
\end{align*}
where we used that owing to condition \ref{cond:discrete-growth} and Lemma \ref{lemma:interpolation-inequality}, we have $\Vert A_h^{1/2} x \Vert_X \leq Ch^{-r/2} \Vert x \Vert_X$ for any $x \in X_h$. For the second term of $(ii)$, we get
\begin{align*}
    \Vert S(t/2) A^{1/2} \Delta_h A_h S_h(t/2) \pi_h x \Vert_X &\leq \Vert A^{1/2} S(t/2) \Vert_{L(X)} \Vert \Delta_h A_h \pi_h \Vert_{L(X)} \Vert S_h(t/2) \pi_h \Vert_{L(X)} \Vert x \Vert_X \\
    &\leq C e^{-\epsilon t} t^{-1/2} \Vert x \Vert_X.
\end{align*}
Finally for the last term of $(ii)$, 
\begin{align*}
    \Vert \int_{t/2}^t S(t-s) A^{1/2} \Delta_h A_h^2 S_h(s) \pi_h x \, ds \Vert_X &\leq \int_{t/2}^t \Vert A^{1/2} S(t-s) \Vert_{L(X)} \Vert \Delta_h A_h \pi_h \Vert_{L(X)} \Vert A_h S_h(s) \pi_h \Vert_{L(X)} \Vert x \Vert_X \, ds \\
    &\leq C e^{-\epsilon t} \int_{t/2}^t (t-s)^{-1/2} s^{-1} \, ds \, \Vert x \Vert_X \\
    &\leq C e^{-\epsilon t} t^{-1/2} \Vert x \Vert_X.
\end{align*}
Therefore, all inequalities hold, and the proof is finished. 
\end{proof}

The next lemma is another useful semidiscrete error estimate. 
\begin{lemma}\label{lemma:2}
Under the same conditions as in Lemma \ref{lemma:1}, there is $C, \epsilon > 0$ such that
\begin{align*}
    \Vert (\lambda + A)^{-\alpha} (u(t) - u_h(t)) \Vert_X \leq C e^{(\lambda-\epsilon) t} t^{-\theta / r} h^{\theta + \alpha r} \Vert x \Vert_X,
\end{align*}
for $\alpha \in [0,1/2]$ and $\theta \in [0,r-\alpha r]$. 
\end{lemma}

To show this we proceed similarly as we did in the case of Lemma \ref{lemma:1}.
\begin{lemma}\label{lemma:interpolation-3}
    Let $B_{t,h} \in L(X)$ be a collection of bounded linear operators satisfying,
    \begin{align*}
        \Vert B_{t,h} \Vert_{L(X)} \leq C h^{r/2}, \quad \Vert B_{t,h} \Vert_{L(X)} \leq C t^{-1/2} h^r, 
    \end{align*}
    for some $C > 0$ independent of $t, h > 0$. Then, the following estimate holds,
    \begin{align*}
        \Vert B_{t,h} x \Vert \leq C t^{-\theta / r} h^{\theta + r / 2} \Vert x \Vert_X, \quad x \in X,
    \end{align*}
    for $\theta \in [0,r/2]$.
\end{lemma}
\begin{proof}
    We get,
    \begin{align*}
        \Vert B_{t,h} \Vert_{L(X)} &= \Vert B_{t,h} \Vert_{L(X)}^{\beta} \Vert B_{t,h} \Vert_{L(X)}^{1- \beta} \\
        &\leq C ( h^{r/2} )^{\beta} ( t^{-1/2} h^{r} )^{1-\beta} \\
        &\leq C t^{-\theta / r} h^{\theta + r/2},
    \end{align*}
    where we have set $\theta = r(1-\beta) / 2 \in [0, r/2]$. 
\end{proof}

\begin{lemma}\label{lemma:interpolation-4}
    Let $B_{t,h} \in L(X)$ be a collection of bounded linear operators satisfying the inequalities of Lemma \ref{lemma:interpolation} and assume that $(\lambda + A)^{-1/2} B_{t,h}$ satisfies the inequalities of Lemma \ref{lemma:interpolation-3}. Then, 
    \begin{align*}
        \Vert (\lambda + A)^{-\alpha} B_{t,h} x \Vert_X \leq C t^{-\theta / r} h^{\theta + \alpha r / 2} \Vert x \Vert_X,  
    \end{align*}
    where $\alpha \in [0,1/2]$, $\theta \in [0, r-\alpha r]$.  
\end{lemma}
\begin{proof}
We use Lemma \ref{lemma:interpolation-inequality} on the inequalities from Lemma \ref{lemma:interpolation} with $\rho = 0$ and Lemma \ref{lemma:interpolation-3}: for $\alpha \in [0,1/2]$,
\begin{align*}
    \Vert (\lambda + A)^{-\alpha} B_{t,h} \Vert_{L(X)} &\leq C \Vert (\lambda + A)^{-1/2} B_{t,h} \Vert_{L(X)}^{2 \alpha} \Vert B_{t,h} \Vert_{L(X)}^{1 - 2\alpha} \\
    &\leq C ( t^{-\theta_1 / r} h^{\theta_1 + r / 2})^{2 \alpha} ( t^{-\theta_2 / r} h^{\theta_2})^{1 - 2\alpha} \\
    &\leq C t^{-\theta / r } h^{\theta + \alpha r}, 
\end{align*}
where $\theta_1 \in [0,r/2], \theta_2 \in [0, r]$, $\theta = 2 \alpha \theta_1 + (1 - 2\alpha) \theta_2 \in [0, r-\alpha r]$. 
\end{proof}

\begin{proof}[Proof of Lemma \ref{lemma:2}]
    The proof is similar to the proof of Theorem \ref{theorem:1} and Lemma \ref{lemma:1}. Also here we can without loss of generality and for ease of notation assume $\lambda = 0$. Owing to Lemma \ref{lemma:interpolation-4}, it suffices to only show that the inequalities of Lemma \ref{lemma:interpolation-3} holds for the operator $e^{\epsilon t} A^{-1/ 2} (S(t) - S_h(t) \pi_h)$ (with $\epsilon$ as in Lemma \ref{lemma:1}) since we have the remaining inequalities by the proof of Theorem \ref{theorem:1}. Similarly as in the proof of Lemma \ref{lemma:1} we set $\Delta_h := A^{-1} - A_h^{-1} \pi_h$, and $e_h(t) := S(t) x - S_h(t) \pi_h x$ for some $x \in X$. 
    
    By the proof of Lemma \ref{lemma:1}, we have (using that $A^{-1/2}$ is bounded, to pass it under the integral sign), 
    \begin{align*}
        A^{-1/2} e_h(t) = A^{-1/2} S(t) (I - \pi_h ) x + \int_0^t A^{-1/2} A S(t-s) \Delta_h A_h S_h(s) \pi_h x \, ds = (I) + (II).
    \end{align*}
    For the first term, we have by Lemma \ref{lemma:analytic-semigroup} and condition \ref{cond:abstract-solution-operator-2}
    \begin{align*}
        \Vert (I) \Vert_X &= \Vert S(t) A^{1/2} (A^{-1} - A^{-1} \pi_h) x \Vert_X \\
        &= \Vert S(t) A^{1/2} (\Delta_h - \Delta_h \pi_h) x \Vert_X \\
        &\leq \Vert S(t) \Vert_{L(X)} \Vert A^{1/2} (\Delta_h - \Delta_h \pi_h) x \Vert_X \\
        &\leq C e^{-\epsilon t} h^{r/2} \Vert x \Vert_X, 
    \end{align*}
    Further, we have by condition \ref{cond:abstract-solution-operator}
    \begin{align*}
        \Vert (I) \Vert_X &\leq \Vert A^{1/2} S(t) (A^{-1} - A^{-1} \pi_h) x \Vert_X \\
        &\leq \Vert A^{1/2} S(t) \Vert_{L(X)} \Vert \Delta_h - \Delta_h \pi_h \Vert_{L(X)} \Vert x \Vert_X \\
        &\leq C e^{-\epsilon t} t^{-1/2} h^r \Vert x \Vert_X.
    \end{align*}
    This finishes the $(I)$-term. 

    For the $(II)$-term, we start by showing that $\Vert (II) \Vert_X \leq C e^{-\epsilon t} h^{r/2} \Vert x \Vert_X$. By the proof of Lemma \ref{lemma:1}, we have,
    \begin{align*}
        (II) &= \bigg(-A^{1/2} S(t/2) \Delta_h S_h(t/2) \pi_h x + S(t) A^{1/2} \Delta_h \pi_h x + \int_0^{t/2} A^{3/2} S(t-s) \Delta_h S_h(s) \pi_h x \, ds \bigg)  \\
        &\quad + \int_{t/2}^t A^{1/2} S(t-s) \Delta_h A_h S_h(s) \pi_h x \, ds \\
        &= (i) + (ii).
    \end{align*}
    For the two first terms of $(i)$ we get using condition \ref{cond:abstract-solution-operator-2}, Lemma \ref{lemma:analytic-semigroup} and \ref{lemma:discrete-semigroup-smoothing}, 
    \begin{align*}
        \Vert S(t/2) A^{1/2} \Delta_h S(t/2) \pi_h x \Vert_X &\leq \Vert S(t/2) \Vert_{L(X)} \Vert A^{1/2} \Delta_h \Vert_{L(X)} \Vert S_h(t/2) \pi_h x \Vert_X \Vert \leq C e^{-\epsilon t} h^{r/2} \Vert x \Vert_X, \\
        \Vert S(t) A^{1/2} \Delta_h \pi_h x \Vert_X &\leq \Vert S(t) \Vert_{L(X)} \Vert A^{1/2} \Delta_h \Vert_{L(X)} \Vert \pi_h x \Vert_X \leq C e^{-\epsilon t} h^{r/2} \Vert x \Vert_X.
    \end{align*}
    For the integral term of $(i)$, we get by Lemma \ref{lemma:analytic-semigroup}, \ref{lemma:discrete-semigroup-smoothing} and condition \ref{cond:abstract-solution-operator-2},
    \begin{align*}
        \Vert \int_0^{t/2} A^{3/2} S(t-s) \Delta_h S_h(s) \pi_h x \, ds \Vert_X &\leq \int_0^{t/2} \Vert A S(t-s) \Vert_{L(X)} \Vert A^{1/2} \Delta_h \Vert_{L(X)} \Vert S_h(s) \pi_h x \Vert_X \, ds \\
        &\leq C e^{-\epsilon t} \int_0^{t/2} (t-s)^{-1} h^{r/2} \Vert x \Vert_X \, ds \\
        &\leq C e^{-\epsilon t} h^{r/2} \Vert x \Vert_X. 
    \end{align*}
    Finally for $(ii)$ we get by similar arguments
    \begin{align*}
        \Vert \int_{t/2}^t A^{1/2} S(t-s) \Delta_h A_h S_h(s) \pi_h x \, ds \Vert_X &\leq \int_{t/2}^t \Vert S(t-s) \Vert_{L(X)} \Vert A^{1/2} \Delta_h \Vert_{L(X)} \Vert A_h S_h(s) \pi_h \Vert_{L(X)} \Vert x \Vert_X \, ds \\
        &\leq C e^{-\epsilon t} \int_{t/2}^t s^{-1} h^{r/2} \Vert x \Vert_X \, ds \\
        &\leq C e^{-\epsilon t} h^{r/2} \Vert x \Vert_X. 
    \end{align*}
    
    We now show that $\Vert (II) \Vert_X \leq C e^{-\epsilon t} t^{-1/2} h^r \Vert x \Vert_X$. For the first two terms of $(i)$ above, we get using Lemma \ref{lemma:analytic-semigroup} and condition \ref{cond:abstract-solution-operator},
    \begin{align*}
        \Vert S(t/2) A^{1/2} \Delta_h S(t/2) \pi_h x \Vert_X &\leq \Vert A^{1/2} S(t/2) \Vert_{L(X)} \Vert \Delta_h \Vert_{L(X)} \Vert S(t/2) x \Vert_X \Vert \leq C e^{-\epsilon t} t^{-1/2} h^r \Vert x \Vert_X \\
        \Vert S(t) A^{1/2} \Delta_h \pi_h x \Vert_X &\leq \Vert A^{1/2} S(t) \Vert_{L(X)} \Vert \Delta_h \Vert_{L(X)} \Vert \pi_h x \Vert_X \leq C e^{-\epsilon t} t^{-1/2} h^r \Vert x \Vert_X.
    \end{align*}
    For the integral term of $(i)$, we get,
    \begin{align*}
        \Vert \int_0^{t/2} A^{3/2} S(t-s) \Delta_h S_h(s) \pi_h x \, ds \Vert_X &\leq \int_0^{t/2} \Vert A^{3/2} S(t-s) \Vert_{L(X)} \Vert \Delta_h \Vert_{L(X)} \Vert S_h(s) \pi_h x \Vert_X \, ds \\
        &\leq C e^{-\epsilon t} \int_0^{t/2} (t-s)^{-3/2} h^r \Vert x \Vert_X \, ds \\
        &\leq C e^{-\epsilon t} t^{-1/2} h^r \Vert x \Vert_X. 
    \end{align*}
    Finally, for $(ii)$, we get,
    \begin{align*}
        \Vert \int_{t/2}^t A^{1/2} S(t-s) \Delta_h A_h S_h(s) \pi_h x \, ds \Vert_X &\leq \int_{t/2}^t \Vert A^{1/2} S(t-s) \Vert_{L(X)} \Vert \Delta_h \Vert_{L(X)} \Vert A_h S_h(s) \pi_h \Vert_{L(X)} \Vert x \Vert_X \, ds \\
        &\leq C e^{-\epsilon t} \int_{t/2}^t (t-s)^{-1/2} s^{-1} h^r \Vert x \Vert_X \, ds \\
        &\leq C e^{-\epsilon t} t^{-1/2} h^r \Vert x \Vert_X.
    \end{align*}
    This finishes the proof. 
\end{proof}

The next lemma is another useful semidiscrete error estimate, but accomodating less regular data than that of Theorem \ref{theorem:1}, and Lemma \ref{lemma:1} and \ref{lemma:2}. In order to show it, we need two additional conditions.
\begin{description}
    \myitem{(A8)}\label{cond:negative-rho} For $\lambda$ as in condition \ref{cond:sectorial} there is some $C > 0$, such that,
    \begin{align*}
        \Vert (\lambda + A_h)^{-1/2} \pi_h (\lambda + A)^{1/2} x \Vert_X \leq C \Vert x \Vert_X,
    \end{align*}
    for any $x \in X$. \medskip 
    \myitem{(A9)}\label{cond:abstract-solution-operator-3} For the same $\lambda$ as in condition \ref{cond:sectorial}, there is some $C > 0$ such that, 
    \begin{align*}
        \Vert ((\lambda + A)^{-1} - (\lambda + A_h)^{-1} \pi_h ) (\lambda + A)^{1/2} x \Vert_X \leq C h^{r/2} \Vert x \Vert_X, 
    \end{align*}
    for any $x \in X$. 
\end{description}

\begin{lemma}\label{lemma:negative-rho}
    Suppose the conditions of Assumption \ref{assumption:abstract-model}, \ref{cond:discrete-growth}, \ref{cond:negative-rho} and \ref{cond:abstract-solution-operator-3} holds. Then, there is $C, \epsilon > 0$ such that
    \begin{align*}
        \Vert u(t) - u_h(t) \Vert_X \leq C e^{(\lambda - \epsilon) t} t^{-\theta / r + \rho / r} h^{\theta} \Vert (\lambda + A)^{\rho / r} x \Vert_X,
    \end{align*}
    for $\theta \in [0,r]$ and $\rho \in [-r/2, 0] \cap [-r + \theta, 0]$.
\end{lemma}

To show this we need the following lemma. 
\begin{lemma}\label{lemma:interpolation-negative-rho}
    Let $B_{t,h} \in L(X)$ be a collection of bounded linear operators, satisfying, 
    \begin{align*}
    &\Vert B_{t,h} \Vert_{L(X)} \leq C, \quad &&\Vert B_{t,h} \Vert_{L(X)} \leq C t^{-1} h^r, \\
    &\Vert B_{t,h} (\lambda + A)^{1/2} \Vert_{L(X)} \leq C t^{-1/2}, \quad &&\Vert B_{t,h} (\lambda + A)^{1/2} \Vert_{L(X)} \leq C t^{-1} h^{r/2}, 
    \end{align*}
    for some $C > 0$. Then, the following estimate holds,
    \begin{align*}
        \Vert B_{t,h} x \Vert \leq C t^{-\theta / r + \rho / r} h^{\theta} \Vert (\lambda + A)^{\rho / r} x \Vert_X, \quad x \in D((\lambda + A)^{\rho / r}),
    \end{align*}
    for $\theta \in [0,r]$, and $\rho \in [-r/2,0] \cap [-r + \theta, 0]$.
\end{lemma}
\begin{proof}
    We have,
    \begin{align*}
        \Vert B_{t,h} (\lambda + A)^{\alpha} \Vert_{L(X)} \leq C t^{-\alpha}, \quad \Vert B_{t,h} (\lambda + A)^{\alpha} \Vert_{L(X)} \leq C t^{-1} h^{(1-\alpha) r},
    \end{align*}
    for $\alpha \in \{0, 1/2\}$, and by Lemma \ref{lemma:interpolation-inequality} it holds for $\alpha \in [0,1/2]$. This gives, 
    \begin{align*}
        \Vert B_{t,h} (\lambda + A)^{\alpha} \Vert_{L(X)} \leq (C t^{-\alpha})^{1-\gamma} (t^{-1} h^{(1-\alpha) r})^{\gamma},
    \end{align*}
    for some $\gamma \in [0,1]$. The inequality follows by setting $\theta = (1-\alpha) \gamma r$, and $\rho = -\alpha r$. 
\end{proof}

\begin{proof}[Proof of Lemma \ref{lemma:negative-rho}]
The proof is similar to that of Theorem \ref{theorem:1}, Lemma \ref{lemma:1} and \ref{lemma:2}. We can by the same arguments as those in the beginning of the proof of Theorem \ref{theorem:1} assume $\lambda = 0$. Owing to Lemma \ref{lemma:interpolation-negative-rho} and the proof of Theorem \ref{theorem:1}, we only need to show the inequalities,
\begin{align*}
    \Vert (S(t) - S_h(t) \pi_h) x \Vert_X &\leq C e^{-\epsilon t} t^{-1/2} \Vert A^{-1/2} x \Vert_X, \\ \Vert (S(t) - S_h(t) \pi_h) x \Vert_X &\leq C e^{-\epsilon t} t^{-1} h^{r/2} \Vert A^{-1/2} x \Vert_X.
\end{align*}
Due to condition \ref{cond:negative-rho}, we have using Lemma \ref{lemma:analytic-semigroup} and \ref{lemma:discrete-semigroup-smoothing},
\begin{align*}
    \Vert (S(t) - S_h(t) \pi_h) x \Vert_X &= \Vert (S(t) - S_h(t) \pi_h) A^{1/2} A^{-1/2} x \Vert_X \\
    &\leq \bigg( \Vert A^{1/2} S(t) \Vert_{L(X)} + \Vert S_h(t) A_h^{1/2} A_h^{-1/2} \pi_h A^{1/2} \Vert_{L(X)} \bigg) \Vert A^{-1/2} x \Vert_X \\
    &\leq C e^{-\epsilon t} t^{-1/2} \Vert A^{-1/2} x \Vert_X. 
\end{align*}

It remains to only show the latter inequality. Also here we define $\Delta_h := A^{-1} - A_h^{-1} \pi_h$. Note that from the proof of Theorem \ref{theorem:1}, we have,
\begin{align*}
    S(t) x - S_h(t) \pi_h x &= (I - \pi_h) S(t) x \\
    &\qquad - S_h(t/2) A_h \pi_h \Delta_h S(t/2) x + S_h(t) A_h \pi_h \Delta_h x + \int_0^{t/2} A_h^2 S_h(t-s) \pi_h \Delta_h S(s) x \, ds \\
    &\qquad + \pi_h \Delta_h A S(t) x - S_h(t/2) \pi_h \Delta_h A S(t/2) x + \int_{t / 2}^t S_h(t-s) \pi_h \Delta_h A^2 S(s) x \, ds.
\end{align*}
For the first term, note that the identity $(I - \pi_h) A^{-1} = \Delta_h - \Delta_h \pi_h$, condition \ref{cond:abstract-solution-operator} and the boundedness of $\pi_h$ implies $\Vert (I - \pi_h) A^{-\alpha} \Vert_{L(X)} \leq C h^{r\alpha}$ for $\alpha \in \{ 0,1\}$, and so by Lemma \ref{lemma:interpolation-inequality} it also holds for $\alpha = 1/2$. Therefore, using Lemma \ref{lemma:analytic-semigroup},
\begin{align*}
    \Vert (I - \pi_h ) S(t) x \Vert_X &= \Vert (I - \pi_h) A^{-1/2} A S(t) A^{-1/2} x \Vert_X \\
    &\leq \Vert (I - \pi_h) A^{-1/2} \Vert_{L(X)} \Vert A S(t) \Vert_{L(X)} \Vert A^{-1/2} x \Vert_X \\
    &\leq C e^{-\epsilon t} t^{-1} h^{r/2} \Vert A^{-1/2} x \Vert_X.
\end{align*}
For the second term, we have
\begin{align*}
    \Vert S_h(t/2) A_h \pi_h \Delta_h S(t/2) x \Vert_X &\leq \Vert A_h^{1/2} S_h(t / 2) \pi_h \Vert_{L(X)} \Vert A_h^{1/2} \pi_h \Vert_{L(X)} \Vert \Delta_h \Vert_{L(X)} \Vert A^{1/2} S(t/2) \Vert_{L(X)} \Vert A^{-1/2} x \Vert_X \\
    &\leq C e^{-\epsilon t} t^{-1/2} h^{-r/2} h^r t^{-1/2} \Vert A^{-1/2} x \Vert_X \\
    &\leq C e^{-\epsilon t} t^{-1} h^{r/2} \Vert A^{-1/2} x \Vert_X,
\end{align*}
where we used condition \ref{cond:discrete-growth} combined with Lemma \ref{lemma:interpolation-inequality} to get $\Vert A_h^{1/2} \pi_h \Vert_{L(X)} \leq C h^{-r/2}$. Moving on to the third term, we find
\begin{align*}
    \Vert S_h(t) A_h \pi_h \Delta_h x \Vert_X &\leq \Vert A_h S_h(t) \pi_h \Vert_{L(X)} \Vert \Delta_h A^{1/2} \Vert_{L(X)} \Vert A^{-1/2} x \Vert_X \\
    &\leq C e^{-\epsilon t} t^{-1} h^{r/2} \Vert A^{-1/2} x \Vert_X,
\end{align*}
where we used Lemma \ref{lemma:discrete-semigroup-smoothing} and condition \ref{cond:abstract-solution-operator-3}. For the fourth term, we have using Lemma \ref{lemma:discrete-semigroup-smoothing}, condition \ref{cond:abstract-solution-operator} and \ref{cond:discrete-growth},
\begin{align*}
    \Vert \int_0^{t/2} A_h^2 S_h(t-s) \pi_h \Delta_h S(s) x \, ds \Vert_X &\leq \int_0^{t/2} \Vert A_h^{3/2} S_h(t-s) \pi_h \Vert_{L(X)} \Vert A_h^{1/2} \Delta_h \Vert_{L(X)} \\
    &\qquad \times \Vert A^{1/2} S(s) \Vert_{L(X)} \Vert A^{-1/2} x \Vert_X \, ds \\
    &\leq C e^{-\epsilon t} \int_0^{t/2} (t-s)^{-3/2} s^{-1/2} h^{r/2} \Vert A^{-1/2} x \Vert_X \, ds \\
    &\leq C e^{-\epsilon t} t^{-1} h^{r/2} \Vert A^{-1/2} x \Vert_X. 
\end{align*}
Moving on to the fifth term, we have by condition \ref{cond:abstract-solution-operator-3} and Lemma \ref{lemma:analytic-semigroup},
\begin{align*}
    \Vert \pi_h \Delta_h A S(t) x \Vert_X &\leq \Vert \Delta_h A^{1/2} \Vert_{L(X)} \Vert A S(t) \Vert_{L(X)} \Vert A^{-1/2} x \Vert_X \\
    &\leq C e^{-\epsilon t} t^{-1} h^{r/2} \Vert A^{-1/2} x \Vert_X.
\end{align*}
Continuing on, we have by the same arguments in addition to Lemma \ref{lemma:discrete-semigroup-smoothing},
\begin{align*}
    \Vert S_h(t/2) \pi_h \Delta_h A S(t/2) x \Vert_X &\leq \Vert S_h(t/2) \pi_h \Vert_{L(X)} \Vert \Delta_h A^{1/2} \Vert_{L(X)} \Vert A S(t) \Vert_{L(X)} \Vert A^{-1/2} x \Vert_X \\
    &\leq C e^{-\epsilon t} t^{-1} h^{r/2} \Vert A^{-1/2} x \Vert_X,
\end{align*}
and for the last term, we get by similar arguments,
\begin{align*}
    \Vert \int_{t/2}^t S_h(t-s) \pi_h \Delta_h A^2 S(s) x \, ds \Vert_X &\leq \int_{t/2}^t \Vert S_h(t-s) \pi_h \Vert_{L(X)} \Vert \Delta_h A^{1/2} \Vert_{L(X)} \Vert A^2 S(s) \Vert_{L(X)} \Vert A^{-1/2} x \Vert_X \, ds \\
    &\leq C e^{-\epsilon t} \int_{t/2}^t s^{-2} h^{r/2} \Vert A^{-1/2} x \Vert_X \, ds \\
    &\leq C e^{-\epsilon t} t^{-1} h^{r/2} \Vert A^{-1/2} x \Vert_X. 
\end{align*}
This finishes the proof. 
\end{proof}

If the conditions of Lemma \ref{lemma:negative-rho} are satisfied, we can get a similar fully discrete estimate, but for $\rho < 0$. To show this, we need the following lemma.

\begin{lemma}\label{lemma:fully-discrete-semigroup-approximation-3}
    Let the conditions of Lemma \ref{lemma:negative-rho} hold. Then, there is $C, c, \epsilon > 0$ such that
    \begin{align*}
        \Vert (S_h(t_n) - S_{h, \Delta t}(t_n) )\pi_h x \Vert_X \leq C e^{c(\lambda-\epsilon) t_n} t_n^{-\theta / r + \rho/r } \Delta t^{\theta / 2} \Vert (\lambda + A)^{\rho/r} x \Vert_X,
    \end{align*}
    for $\theta \in [0,r]$ and $\rho \in [-r/2, 0] \cap [-r + \theta,0]$. 
\end{lemma}
\begin{proof}
    From condition \ref{cond:negative-rho} we gather that
    \begin{align*}
        &\Vert (S_h(t_n) - S_{h,\Delta t}(t_n)) \pi_h (\lambda + A)^{1/2} \Vert_{L(X)} \\
        &\qquad \leq \Vert (S_h(t_n) - S_{h,\Delta t}(t_n)) (\lambda + A_h)^{1/2} \Vert_{L(X_h)} \Vert (\lambda + A_h)^{-1/2} \pi_h (\lambda + A)^{-1/2} \Vert_{L(X)} \\
        &\qquad \leq C \Vert (S_h(t_n) - S_{h,\Delta t}(t_n)) (\lambda + A_h)^{1/2} \Vert_{L(X_h)}.
    \end{align*}
    Thus, by Lemma \ref{lemma:fully-discrete-inequalities-1}, \ref{lemma:fully-discrete-inequalities-3}, \ref{lemma:fully-discrete-inequalities-4} and \ref{lemma:fully-discrete-inequalities-5}, combined with Lemma \ref{lemma:interpolation-inequality} we have for $\alpha \in [0,1/2]$,
    \begin{align*}
        \Vert (S_h(t_n) - S_{h, \Delta t}(t_n)) (\lambda + A)^{\alpha} \Vert_{L(X)} &\leq C e^{c(\lambda - \epsilon) t_n} t_n^{-\alpha}, \\
        \Vert (S_h(t_n) - S_{h, \Delta t}(t_n)) (\lambda + A)^{\alpha} \Vert_{L(X)} &\leq C e^{c(\lambda - \epsilon) t_n} t_n^{-1} \Delta t^{(1-\alpha)r},
    \end{align*}
    which gives for some $\gamma \in [0,1]$, 
    \begin{align*}
        \Vert (S_h(t_n) - S_{h, \Delta t}(t_n)) (\lambda + A)^{\alpha} \Vert_{L(X)} \leq C e^{c(\lambda - \epsilon) t_n} (t_n^{-\alpha})^{1-\gamma} (t_n^{-1} h^{(1-\alpha)})^{\gamma}.
    \end{align*}
    Setting $\theta = (1-\alpha) \gamma r$, and $\rho = -\alpha r$ gives the inequality. 
\end{proof}

The next theorem extends Theorem \ref{theorem:fully-discrete-semigroup-approximation-2} to the case of $\rho < 0$. 
\begin{theorem}\label{theorem:fully-discrete-semigroup-approximation-4}
    Suppose the conditions of Assumption \ref{assumption:abstract-model}, \ref{cond:discrete-growth}, \ref{cond:negative-rho} and \ref{cond:abstract-solution-operator-3} holds. Then, there is $C, c, \epsilon > 0$ such that 
    \begin{align*}
        \Vert (S(t) - S_{h,\Delta t}(t) \pi_h ) x \Vert_X \leq C e^{c(\lambda-\epsilon) t} t^{-\theta / r + \rho / r} (h^{\theta} + \Delta t^{\theta / 2}) \Vert (\lambda + A)^{\rho / r} x\Vert_X,
    \end{align*}
    for $\theta \in [0,r]$ and $\rho \in [-r/2, 0] \cap [-r + \theta, 0]$. 
\end{theorem}
\begin{proof}
    The proof is almost the same as that of Theorem \ref{theorem:fully-discrete-semigroup-approximation-2}, but now we have to use Lemma \ref{lemma:negative-rho} and Lemma \ref{lemma:fully-discrete-semigroup-approximation-3}: for $t \in (t_n, t_{n+1}]$, we decompose the error as follows, 
    \begin{align*}
        &S(t) - S(t_{n+1}) \\
        &\qquad + S(t_{n+1}) -  S_h(t_{n+1}) \pi_h \\
        &\qquad + S_h(t_{n+1}) \pi_h -  S_{h,\Delta t} (t_{n+1}) \pi_h \\
        &\qquad + S_{h,\Delta t} (t_{n+1}) \pi_h -  S_{h,\Delta t} (t) \pi_h.
    \end{align*} 
    For the first term, we have for some $\theta \in [0,r], \rho \in [-r/2, 0] \cap [-r + \theta, 0]$,
    \begin{align*}
        \Vert (S(t) - S(t_{n+1})) x \Vert_X &= \Vert (I - S(t_{n+1} - t) ) S(t) x \Vert_X \\
        &= \Vert (I - S(t_{n+1} - t)) (\lambda + A)^{-\theta / r} (\lambda + A)^{\theta / r - \rho / r}  S(t) (\lambda + A)^{\rho / r} x \Vert_X \\
        &\leq \Vert (I - S(t_{n+1} - t)) (\lambda + A)^{-\theta / r} \Vert_{L(X)} \Vert (\lambda + A)^{\theta / r - \rho / r}  S(t) \Vert_{L(X)} \Vert (\lambda + A)^{\rho / r} x \Vert_X.
    \end{align*}
    As in the proof of Theorem \ref{theorem:fully-discrete-semigroup-approximation-2}, we have by Lemma \ref{lemma:analytic-semigroup},
    \begin{align*}
        \Vert (I - S(t_{n+1} - t)) (\lambda + A)^{-\theta / r} \Vert_{L(X)} \leq C \Delta t^{\theta / r} \quad \text{and} \quad \Vert (\lambda + A)^{\theta / r - \rho / r} S(t) x \Vert_X \leq C e^{(\lambda-\epsilon) t} t^{-\theta / r + \rho / r} \Vert x \Vert_X. 
    \end{align*}
    Combining these inequalities finishes the first term. For the two next terms, we now use Lemma \ref{lemma:negative-rho} and Lemma \ref{lemma:fully-discrete-semigroup-approximation-3}. For the last term, we also here note that,
    \begin{align*}
        S_{h,\Delta t} (t_{n+1}) \pi_h - S_{h,\Delta t} (t) \pi_h = 0, \quad t \in (t_n,t_{n+1}],
    \end{align*}
    and arguing as in the proof of Theorem \ref{theorem:fully-discrete-semigroup-approximation-2} we may replace $t_{n+1}^{-\theta / r + \rho / r}$ with $t^{-\theta / r + \rho / r}$. This finishes the proof. 
\end{proof}
\section{Example: surface finite element approximation of a parabolic equation}\label{section:example}

To relate the abstract conditions of Assumption \ref{assumption:abstract-model} to differentiation operators, Lemma \ref{lemma:variational-semigroup} is key. To that end, let $V, H$ be Hilbert spaces with $V \subseteq H$, densely and continuously with inclusion $j : V \to H$, and let $a : V \times V \to \mathbb{C}$ be a coercive and continuous sesquilinear form on $V$. Recall that $a$ is continuous if for some $C > 0$,
\begin{align*}
    \vert a(u, v) \vert \leq C \Vert u \Vert_V \Vert v \Vert_V, \quad u, v \in V,
\end{align*}
and it is coercive if for some $c > 0$,
\begin{align*}
    \mathrm{Re}(a(u, u)) \geq c \Vert u \Vert_V^2, \quad u \in V.
\end{align*}
By Lax--Milgram theorem, the operator,
\begin{align*}
    A : V \to V^*, \quad A u = a(u, \cdot),
\end{align*}
is an isomorphism. Since $(\varphi, j \cdot)_H \in V^*$ for any $\varphi \in H$, we can also define (without changing notation), an operator $A : D(A) \to H$, $D(A) \subseteq V$ by identifying $a(u, \cdot)$ with $\varphi \in H$, in the case when,
\begin{align*}
    a(u, \cdot) = (\varphi, j \cdot)_H, \quad u \in D(A).
\end{align*}
The following lemma asserts that $A : D(A) \to H$ is sectorial.
\begin{lemma}\label{lemma:variational-semigroup}
    Let $A : D(A) \to H, \ D(A) \subseteq V$ be defined by $(A u, j v)_H = a(u, v)$ for any $v \in V$, where $a(\cdot, \cdot)$ is a coercive and continuous sesquilinear form on $V$. Set $M' := 1 + C / c$, where $C,c$ are the continuity and coercivity constants of $a(\cdot, \cdot)$, respectively. Then: \smallskip
    \begin{enumerate}
        \item for any $\delta \in (\pi/2-\sin^{-1}(M'^{-1}), \pi / 2)$, we have \begin{align*}
        \sigma(A) \subseteq \Sigma_{\delta} := \{ z \in \mathbb{C} \setminus \{ 0 \}, \ \vert \mathrm{arg}(z) \vert < \delta \},
        \end{align*}
        \item and for any $z \notin \Sigma_{\delta}$, we have, 
        \begin{align}
            \Vert (z - A)^{-1} x \Vert_H \leq \frac{M}{\vert z \vert} \Vert x \Vert_H, \quad \text{where} \quad M = \frac{M' \cos(\pi / 2 - \delta)}{1 - M' \sin(\pi / 2 - \delta)} \geq M'. 
        \end{align}
    \end{enumerate}
\end{lemma}
\begin{proof}
    See Chapter 2 in \cite{yagi}, and in particular Theorem 2.1, Equation 2.4 and 2.5.
\end{proof}

With Lemma \ref{lemma:variational-semigroup} at hand, we are ready to connect the results of the previous section to results related to the numerical approximation of parabolic partial differential equations. We consider the model problem, 
\begin{align*}
    \dot{u} = -L u, \quad u(0) = u_0 \in L^2(\Gamma),
\end{align*}
where (we refer to Appendix \ref{app:hypersurfaces} for the appropriate definitions):
\begin{enumerate}
    \item $\Gamma \subseteq \mathbb{R}^3$ is a compact (and therefore boundaryless) $2$-dimensional $C^3$-surface,
    \medskip
    \item $V$ is a closed subspace of $H^1(\Gamma)$, and $H$ is $L^2(\Gamma)$, 
    \medskip 
    \item $L := \mathcal{L} - I$, where $\mathcal{L}$ is defined as in Lemma \ref{lemma:variational-semigroup}, using the (coercive and continuous) sesquilinear form, 
    \begin{align*}
        a(u, v) := \int_{\Gamma} \nabla_{\Gamma} u \cdot \nabla_{\Gamma} \overline{v} + u \overline{v} \, d\sigma, 
    \end{align*}
    on $V$, where $\nabla_{\Gamma}$ are tangential derivatives on $\Gamma$, and $\sigma$ is surface measure on $\Gamma$, and $\overline{v}$ is the complex conjugate of $v$.
\end{enumerate}
By Lemma \ref{lemma:elliptic-regularity} we deduce that $D(\mathcal{L}) = H^2(\Gamma) \cap V$, and for $u \in H^2(\Gamma) \cap V$, $L u = \Delta_{\Gamma} u$ and $\mathcal{L} u = (I - \Delta_{\Gamma}) u$ for a.e. $x \in \Gamma$. The latter identity combined with Lemma \ref{lemma:elliptic-regularity}, gives for any $f \in L^2(\Gamma)$, 
\begin{align}\label{eq:example-elliptic-regularity-estimate}
    \Vert \mathcal{L}^{-1} f \Vert_{H^2(\Gamma)} \leq C \Vert f \Vert_{L^2(\Gamma)},
\end{align}
which is an elliptic regularity estimate for solutions, $u$, of
\begin{align*}
    \mathcal{L} u = f. 
\end{align*}
By Lemma \ref{lemma:variational-semigroup} it is clear that $-L$ generates an analytic semigroup on $L^2(\Gamma)$---in the setting of Assumption \ref{assumption:abstract-model}, \ref{cond:sectorial} we have $L = A$ and $\mathcal{L} = \lambda + A$, with $\lambda = 1$. 

In order to approximate $\mathcal{L}$ numerically, we use the surface finite element approximation proposed in \cite{1988-dziuk}. To that end, let $\Gamma_h := \bigcup_{\tau \in \mathcal{T}_h} \tau \subseteq \mathbb{R}^3$ be a $2$-dimensional surface, consisting of a collection, $\mathcal{T}_h$, of planar simplices, $\tau$, whose vertices lie on $\Gamma$, and where $h := \max_{\tau \in \mathcal{T}_h} \text{diam}(\tau)$. The collection $\mathcal{T}_h$ is regular in the sense that for some $C > 0$,
\begin{align*}
    C^{-1} h \leq 2 \rho(\tau) \leq \text{diam}(\tau) \leq 2 r(\tau) \leq C h, \quad \tau \in \mathcal{T}_h,
\end{align*}
where $\rho, r$ are the radii of the incircle and circumcircle, respectively. Let $V_h \subseteq C(\overline{\Gamma}_h)$, be the finite dimensional linear space of continuous functions that are affine linear when restricted to $\tau \in \mathcal{T}_h$. We denote the nodal basis of $V_h$ associated to the triangle vertices $\{ x_1, \dots, x_{N_h} \} \in \Gamma_h \cap \Gamma$, by $\varphi_n$, $n = 1, \dots, N_h$. Thus $\varphi_n$ are piece wise affine linear, with $\varphi_n(x_n) = 1$ and $\varphi_n(x_j) = 0$, $j \neq n$. 

To relate functions defined on $\Gamma_h$ to those defined on $\Gamma$, the closest point projection is key. To that end, let $\mathcal{N}_{\epsilon}$ be a tubular neighbourhood of $\Gamma$ (see Appendix \ref{app:hypersurfaces}). The closest point projection, $p : \mathcal{N}_{\epsilon} \to \Gamma$ is defined by,
\begin{align*}
    p(x) := x - d(x) \nabla d(x),
\end{align*}
where $d$ is the signed distance function of $\Gamma$. Since $\Gamma$ is a $C^3$-surface, we have $p \in C^2(\mathcal{N}_{\epsilon})$. For $h$ small enough, we have $\Gamma_h \subseteq \mathcal{N}_{\epsilon}$, and $p$ restricted to $\Gamma_h$ becomes a bijection $\Gamma_h \to \Gamma$. Therefore, for any function $f : \Gamma_h \to \mathbb{R}$, we can define the lift, 
\begin{align}\label{eq:iota_h}
    \iota_h f := f \circ p\vert_{\Gamma_h}^{-1} : \Gamma \to \mathbb{R}.
\end{align}

Our surface finite element approximation of $\mathcal{L}$, denoted $\mathcal{L}_h : V_h \to V_h$, is defined by requiring that,
\begin{align*}
    (\mathcal{L}_h u, v)_{L^2(\Gamma_h)} = a_h(u, v), \quad u, v \in V_h,
\end{align*}
where $a_h : V_h \times V_h \to \mathbb{R}$ is defined by,
\begin{align}\label{eq:approximate-bilinear-form}
    a_h(u, v) = \int_{\Gamma_h} \nabla_{\Gamma_h} u \cdot \nabla_{\Gamma_h} \overline{v} + u \overline{v} \, d\sigma_h.
\end{align}
Here, $\nabla_{\Gamma_h}$, $\sigma_h$ are tangential derivatives and surface measure on $\Gamma_h$, respectively---the tangential derivatives on the discrete surface are defined similarly as on the smooth surface, but only in an almost everywhere sense since the unit normal on the discrete surface is only defined almost everywhere. Similarly, we define,
\begin{align*}
    L_h := \mathcal{L}_h - I.
\end{align*}
Finally we denote by $\mathcal{P}_h : L^2(\Gamma_h) \to V_h$ the $L^2(\Gamma_h)$-orthogonal projection onto $V_h$. We are now ready to list some useful properties of $\Gamma_h$ and $\mathcal{L}_h$. 
\begin{lemma}\label{lemma:sfem}
    Let $\Gamma_h$ be as above. Then the quotient of surface area measure, $\delta_h : \Gamma_h \to [0,\infty)$, defined ($\sigma_h$-a.e.) by,
    \begin{align*}
         \int_{p(U)} d\sigma = \int_{U} \delta_h \, d \sigma_h, \quad U \subseteq \Gamma_h \text{ measurable},
    \end{align*}
    satisfies,
    \begin{align}\label{eq:quotient-h2}
        \Vert 1 - \delta_h \Vert_{L^{\infty}(\Gamma_h)} \leq C h^2,
    \end{align}
    \begin{align}\label{eq:derivative-h1}
        \Vert 1 - \iota_h \delta_h^{-1} \Vert_{L^{\infty}(\Gamma)} + \sum_{i=1}^{3} \Vert \underline{\partial}_i (1 - \iota_h \delta_h^{-1}) \Vert_{L^{\infty}(\Gamma)} \leq C h,
    \end{align}
    and 
    \begin{align}\label{eq:derivative-bounded}
        \Vert \iota_h \delta_h^{-1} \Vert_{L^{\infty}(\Gamma)} + \sum_{i=1}^{3} \Vert \underline{\partial}_i \iota_h \delta_h^{-1} \Vert_{L^{\infty}(\Gamma)} \leq C, 
    \end{align}
    for some $C > 0$ independent of $h$. Here $\underline{\partial}_i$ denotes the $i$'th component of the tangential derivative. 
\end{lemma}
\begin{proof}
    Without loss of generality, assume $x_3 = 0$ on a given $\tau \in \mathcal{T}_h$, and so surface measure on $\tau$ is given by $d \sigma_h = d x_1 d x_2$. For $x \in \tau$, $\delta_h$ is defined by,
    \begin{align}\label{eq:deltah-definition}
        \delta_h^2 := 
        \begin{vmatrix}
            \partial_{x_1} p_2 & \partial_{x_1} p_3 \\
            \partial_{x_2} p_2 & \partial_{x_2} p_3
        \end{vmatrix}^2
        +
        \begin{vmatrix}
            \partial_{x_1} p_1 & \partial_{x_1} p_3 \\
            \partial_{x_2} p_1 & \partial_{x_2} p_3
        \end{vmatrix}^2
        +
        \begin{vmatrix}
            \partial_{x_1} p_1 & \partial_{x_1} p_2 \\
            \partial_{x_2} p_1 & \partial_{x_2} p_2
        \end{vmatrix}^2.
    \end{align}
    The inequality \eqref{eq:quotient-h2} holds by Lemma 4.1 in \cite{2013-elliott}, and from the proof of that lemma we also gather that,
    \begin{align}\label{eq:d-estimates}
        \Vert d \Vert_{L^{\infty}(\tau)} \leq C h^2, \quad \Vert \nu_i \Vert_{L^{\infty}(\tau)} \leq C h, 
    \end{align}
    with $\nu_i = \partial_{x_i} d$, $i = 1, 2$, for some $C > 0$ that only depends on the $C^2(\mathcal{N}_{\epsilon})$-norm of $d$. For any $x \in \mathcal{N}_{\epsilon}$, $j = 1, 2$ and $i, k = 1, 2, 3$, we have
    \begin{align*}
        \partial_{x_j} p_i = \delta_{ij} - \nu_j \nu_i - d \nu_{ij}, \quad \partial_{x_k} \partial_{x_j} p_i = -\nu_{kj} \nu_i - \nu_j \nu_{ik} - \nu_k \nu_{ij} - d \nu_{ijk},
    \end{align*}
    where $\delta_{ij} = 1$ if $i = j$ and $0$ otherwise. Using \eqref{eq:d-estimates} and that $d \in C^3(\mathcal{N}_{\epsilon})$, we get
    \begin{align*}
        \Vert \partial_{x_j} p_i - \delta_{ij} \Vert_{L^{\infty}(\tau)} \leq C h^2, \quad \Vert \partial_{x_k} \partial_{x_j} p_i \Vert_{L^{\infty}(\tau)} \leq C h.
    \end{align*}
    Since $p$ is defined on all of $\mathcal{N}_{\epsilon}$, $\delta_h$ has a natural extension to a tubular neighbourhood of any $\tau \in \mathcal{T}_h$, given by \eqref{eq:deltah-definition}, and which we denote by $\Tilde{\delta}_h$. By \eqref{eq:quotient-h2} we have that $\delta_h \in [1/2, 3/2]$ for all $h$ small, so by the expression \eqref{eq:deltah-definition} for $\delta_h^2$, we must have that for $i = 1, 2, 3$, 
    \begin{align*}
        \Vert \partial_{x_i} \tilde{\delta}_h \Vert_{L^{\infty}(\tau)} = \Vert \frac{1}{2} \tilde{\delta}_h^{-1} \partial_{x_i} \tilde{\delta}_h^2 \Vert_{L^{\infty}(\tau)} \leq \frac{1}{2} \Vert \tilde{\delta}_h^{-1} \Vert_{L^{\infty}(\tau)} \Vert \partial_{x_i} \tilde{\delta}_h^2 \Vert_{L^{\infty}(\tau)} \leq C h,
    \end{align*}
    where the constant $C > 0$ depends only on the $C^3(\mathcal{N}_{\epsilon})$-norm of $d$. The inequality above implies that $\Vert \nabla \tilde{\delta}_h \Vert_{L^{\infty}(\tau)} \leq C h$, which in turn implies that $\Vert \nabla \tilde{\delta}_h \Vert_{L^{\infty}(\Gamma_h)} \leq C h$. From the proof of Lemma 3 in \cite{1988-dziuk} we gather that for any $\tilde{f}$ defined in a neighbourhood of $\tau$ and differentiable, there is $C > 0$ such that for $f = \tilde{f}\vert_{\tau}$,
    \begin{align*}
        \vert ( \nabla_{\Gamma} \iota_h f ) \circ p \vert \leq C \vert \nabla \tilde{f} \vert, \quad x \in \tau.
    \end{align*}
    This gives for $i = 1, 2, 3$,
    \begin{align*}
        \Vert \underline{\partial}_i (1 - \iota_h \delta_h^{-1}) \Vert_{L^{\infty}(\Gamma)} \leq C \Vert \nabla \tilde{\delta}_h^{-1} \Vert_{L^{\infty}(\Gamma_h)} \leq C \Vert \tilde{\delta}_h^{-2} \Vert_{L^{\infty}(\Gamma_h)} \Vert \nabla \tilde{\delta}_h \Vert_{L^{\infty}(\Gamma_h)} \leq C h,
    \end{align*}
    for all $h$ small. Note furthermore that by \eqref{eq:quotient-h2}, 
    \begin{align*}
        \Vert 1 - \iota_h \delta_h^{-1} \Vert_{L^{\infty}(\Gamma)} = \Vert \delta_h^{-1} (1 - \delta_h) \Vert_{L^{\infty}(\Gamma_h)} \leq \Vert \delta_h^{-1} \Vert_{L^{\infty}(\Gamma_h)} \Vert 1 - \delta_h \Vert_{L^{\infty}(\Gamma_h)} \leq C h^2.
    \end{align*}
    Combining these observations gives \eqref{eq:derivative-h1} and \eqref{eq:derivative-bounded}. 
\end{proof}

\begin{lemma}\label{lemma:example-bilinear-form-error}
    Let $a$ and $a_h$ be as above, and let $u, v \in \iota_h V_h$. Then the following estimates holds,
    \begin{align*}
        \vert a(u, v) - a_h(\iota_h^{-1} u, \iota_h^{-1} v) \vert \leq C h^2 \Vert u \Vert_{H^1(\Gamma)} \Vert v \Vert_{H^1(\Gamma)}.
    \end{align*}
    Further, for $u, v \in L^2(\Gamma)$,
    \begin{align*}
        \vert (u, v)_{L^2(\Gamma)} - (\iota_h^{-1} u, \iota_h^{-1} v)_{L^2(\Gamma_h)} \vert \leq C h^2 \Vert u \Vert_{L^2(\Gamma)} \Vert v \Vert_{L^2(\Gamma)}. 
    \end{align*}
\end{lemma}
\begin{proof}
    The first inequality of the lemma follows by Lemma 4.7 in \cite{2013-elliott}. The second inequality follows by \eqref{eq:quotient-h2}, noting that,
    \begin{align*}
        \vert (u, v)_{L^2(\Gamma)} - (\iota_h^{-1} u, \iota_h^{-1} v)_{L^2(\Gamma_h)} \vert = \vert (u, v (1 - \delta_h^{-1}))_{L^2(\Gamma)} \vert \leq \Vert u \Vert_{L^2(\Gamma)} \Vert v \Vert_{L^2(\Gamma)} \Vert 1 - \delta_h^{-1} \Vert_{L^{\infty}(\Gamma_h)}. 
    \end{align*}
\end{proof}

\begin{lemma}\label{lemma:H1-norm-equivalence}
    Let $\Gamma_h$ and $\mathcal{L}_h$ be as above. Then the norms,
    \begin{align*}
        \Vert \cdot \Vert_{L^2(\Gamma)}, \quad \Vert \iota_h^{-1} \cdot \Vert_{L^2(\Gamma_h)},
    \end{align*}
    are equivalent on $L^2(\Gamma)$ with constants independent of $h$, while the norms
    \begin{align*}
        \Vert \cdot \Vert_{H^1(\Gamma)}, \quad \Vert \iota_h^{-1} \cdot \Vert_{H^1(\Gamma_h)},
    \end{align*}
    with $\Vert u \Vert_{H^1(\Gamma_h)}^2 := \Vert u \Vert_{L^2(\Gamma_h)}^2 + \Vert \nabla_{\Gamma_h} u \Vert_{L^2(\Gamma_h)}^2$, are all equivalent on $H^1(\Gamma)$ with constants independent of $h$. 
\end{lemma}
\begin{proof}
    This lemma follows by Lemma 4.2 in \cite{2013-elliott}. 
\end{proof}

\begin{lemma}\label{lemma:example-discrete-growth}
    Let $\Gamma_h$ and $\mathcal{L}_h$ be as above. Then there is $C > 0$ independent of $h$, such that for any $u \in V_h$,
    \begin{align*}
        \Vert \mathcal{L}_h u \Vert_{L^2(\Gamma_h)} \leq C h^{-2} \Vert u \Vert_{L^2(\Gamma_h)}.
    \end{align*}
\end{lemma}
\begin{proof}
    This inequality holds since the mesh is regular. When the mesh is regular, the following inequality holds (see e.g. Theorem 3.2.6 in \cite{cia-fem}, and note that it does not matter that our simplices are embedded in $\mathbb{R}^3$ as opposed to $\mathbb{R}^2$),
    \begin{align}\label{eq:regular-mesh-inequality}
        \Vert u \Vert_{H^1(\Gamma_h)} \leq C h^{-1} \Vert u \Vert_{L^2(\Gamma_h)}, \quad \text{for any } u \in V_h. 
    \end{align}
    This in turn gives for any $u \in V_h$,
    \begin{align*}
        \Vert \mathcal{L}_h u \Vert_{L^2(\Gamma_h)} &= \sup_{\varphi \in V_h, \Vert \varphi \Vert_{L^2(\Gamma_h)} = 1,} (\mathcal{L}_h u, \varphi)_{L^2(\Gamma_h)} \\
        &= \sup_{\varphi \in V_h, \Vert \varphi \Vert_{L^2(\Gamma_h)}} a_h(u, \varphi) \\
        &\leq \Vert u \Vert_{H^1(\Gamma_h)} \sup_{\varphi \in V_h, \Vert \varphi \Vert_{L^2(\Gamma_h)}=1} \Vert \varphi \Vert_{H^1(\Gamma_h)} \\
        &\leq C h^{-2} \Vert u \Vert_{L^2(\Gamma_h)}.
    \end{align*}
\end{proof}

We are now ready to place the surface finite element approximation above into the framework of Assumption \ref{assumption:abstract-model}. To that end, let: \smallskip
\begin{enumerate}
    \item $X := H$. \medskip
    \item $X_h := \iota_h V_h$. \medskip
    \item $\pi_h := \iota_h \mathcal{P}_h \iota_h^{-1} : H \to \iota_h V_h$. \medskip
    \item $A := L : D(L) \to H$. \medskip
    \item $A_h := \iota_h L_h \iota_h^{-1} : \iota_h V_h \to \iota_h V_h$. \smallskip
\end{enumerate}
Note first that condition \ref{cond:sectorial} holds by Lemma \ref{lemma:variational-semigroup} for the operator $A$. The next lemma asserts that condition \ref{cond:discrete-sectorial} holds for the discrete operator, $A_h$. 
\begin{lemma}
    $A_h : X_h \to X_h$ defined as above satisfies condition \ref{cond:discrete-sectorial}. 
\end{lemma}
\begin{proof}
Condition \ref{cond:discrete-sectorial} holds by Lemma \ref{lemma:variational-semigroup} provided we can show that the continuity and coercivity constants of the sesquilinear form, 
\begin{align*}
    (\iota_h \mathcal{L}_h \iota_h^{-1} \cdot, \cdot)_H : \iota_h V_h \times \iota_h V_h \to \mathbb{R},
\end{align*}
on the lifted finite element space $\iota_h V_h$, equipped with the $H^1(\Gamma)$-norm, are independent of $h$.

We first verify that the continuity constant can be chosen independently of $h$. For any $u, v \in \iota_h V_h$, we have,
\begin{align*}
    (\iota_h \mathcal{L}_h \iota_h^{-1} u, v)_{L^2(\Gamma)} &= (\mathcal{L}_h \iota_h^{-1} u, \iota_h^{-1} v)_{L^2(\Gamma_h)} - (\mathcal{L}_h \iota_h^{-1} u, (1-\delta_h) \iota_h^{-1} v)_{L^2(\Gamma_h)} \\
    &= a_h(\iota_h^{-1} u, \iota_h^{-1} v) - (\mathcal{L}_h \iota_h^{-1} u, (1-\delta_h) \iota_h^{-1} v)_{L^2(\Gamma_h)}.
\end{align*}
For the first term, we have,
\begin{align*}
    \vert a_h(\iota_h^{-1} u, \iota_h^{-1} v) \vert \leq C \Vert u \Vert_{H^1(\Gamma)} \Vert v \Vert_{H^1(\Gamma)},
\end{align*}
by using Cauchy--Schwarz, and Lemma \ref{lemma:H1-norm-equivalence}. For the second term, note first that owing to \eqref{eq:quotient-h2} and \eqref{eq:regular-mesh-inequality} we must have $\Vert \mathcal{L}_h^{1/2} f \Vert_{L^2(\Gamma_h)} \leq C h^{-1} \Vert f \Vert_{L^2(\Gamma_h)}$ for $f \in V_h$. Then by Cauchy--Schwarz,
\begin{align*}
    \vert (\mathcal{L}_h \iota_h^{-1} u, (1-\delta_h)\iota_h^{-1} v)_{L^2(\Gamma)} \vert &\leq \Vert \mathcal{L}_h^{1/2} \Vert_{L(V_h)} \Vert \mathcal{L}_h^{1/2} \iota_h^{-1} u \Vert_{L^2(\Gamma_h)} \Vert (1 - \delta_h) \iota_h^{-1} v \Vert_{L^2(\Gamma_h)} \\
    &\leq C h^{-1} \Vert u \Vert_{H^1(\Gamma)} \Vert 1 - \delta_h \Vert_{L^{\infty}(\Gamma_h)} \Vert v \Vert_{L^2(\Gamma)} \\
    &\leq C h \Vert u \Vert_{H^1(\Gamma)} \Vert v \Vert_{H^1(\Gamma)},
\end{align*}
where we used Lemma \ref{lemma:sfem}. Combining the two terms, we see that,
\begin{align*}
    (\iota_h \mathcal{L}_h \iota_h^{-1} u, \iota_h^{-1} u)_{L^2(\Gamma)} \leq C \Vert u \Vert_{H^1(\Gamma)} \Vert v \Vert_{H^1(\Gamma)},
\end{align*}
for some $C > 0$ independent of $h$. This shows that the continuity constant of $((I + A_h)\cdot, \cdot)_{L^2(\Gamma)}$ does not depend on $h$.

It remains to show that the coercivity constant can be chosen independently of $h$. Also here we have for $u \in \iota_h V_h$,
\begin{align*}
    (\iota_h \mathcal{L}_h \iota_h^{-1} u, u)_{L^2(\Gamma)} = a_h(\iota_h^{-1} u, \iota_h^{-1} u) - (\mathcal{L}_h \iota_h^{-1} u, (1-\delta_h) \iota_h^{-1} u)_{L^2(\Gamma_h)},
\end{align*}
and for the first term, we get,
\begin{align*}
    \mathrm{Re}( a_h(\iota_h^{-1} u, \iota_h^{-1} u) ) = \Vert \iota_h^{-1} u \Vert_{H^1(\Gamma_h)}^2 \geq c \Vert u \Vert_{H^1(\Gamma)}^2,   
\end{align*}
by Lemma \ref{lemma:H1-norm-equivalence}. For the second term, we have by the same arguments as in the case of the continuity constant,
\begin{align*}
    \vert (\mathcal{L}_h \iota_h^{-1} u, (1-\delta_h)\iota_h^{-1} u)_{L^2(\Gamma)} \vert &\leq C h \Vert u \Vert_{H^1(\Gamma)}^2,
\end{align*}
so that combining the two terms, yields, 
\begin{align*}
    \mathrm{Re}( (\iota_h \mathcal{L}_h \iota_h^{-1} u, \iota_h^{-1} u)_{L^2(\Gamma)} ) \geq c \Vert u \Vert_{H^1(\Gamma)}^2 - C h \Vert u \Vert_{H^1(\Gamma)}^2.
\end{align*}
It follows that $((I + A_h) \cdot, \cdot)_{L^2(\Gamma)}$ is coercive with constant that does not depend on $h$ for all $h$ small. 
\end{proof}

The next lemma (and a density argument) asserts that condition \ref{cond:negative-rho} holds for $A, A_h$ and $\pi_h$ as above. 
\begin{lemma}\label{lemma:example-negative-rho-condition}
    Let $A, A_h$ and $\pi_h$ be as above. Then
    \begin{align*}
        \Vert (I + A_h)^{-1/2} \pi_h (I + A)^{1/2} u \Vert_{L^2(\Gamma)} \leq C \Vert u \Vert_{L^2(\Gamma)},
    \end{align*}
    for any $u \in V$, and so condition \ref{cond:negative-rho} holds for this choice of $A$, $A_h$ and $\pi_h$.
\end{lemma}
\begin{proof}
    Since for any $u \in \iota_h V_h$, we have
    \begin{align*}
        \Vert u \Vert_H = \sup_{\varphi \in \iota_h V_h, \Vert \varphi \Vert_{L^2(\Gamma)} = 1} (u, \varphi)_{L^2(\Gamma)},
    \end{align*}
    it suffices to show that,
    \begin{align*}
        ((I + A_h)^{-1/2} \pi_h (I + A)^{1/2} u, \varphi)_H \leq C \Vert u \Vert_H,
    \end{align*}
    for some $C > 0$, with $u \in V$ and $\varphi \in \iota_h V_h$ with $\Vert \varphi \Vert_H = 1$. Noting that for any $\alpha \in \mathbb{R}$, $(I + A_h)^{\alpha} = \iota_h \mathcal{L}_h^{\alpha} \iota_h^{-1}$, $\mathcal{L}_h^{\alpha}$ and $\mathcal{L}^{\alpha}$ are selfadjoint (with the $L^2(\Gamma_h)$-inner product on $V_h$, and on $L^2(\Gamma)$, respectively), we get
    \begin{align*}
        ((I + A_h)^{-1/2} \pi_h (I + A)^{1/2} u, \varphi)_H &= \int_{\Gamma} ( \iota_h \mathcal{L}_h^{-1/2} \mathcal{P}_h \iota_h^{-1} \mathcal{L}^{1/2} u ) \overline{\varphi} \, d \sigma \\
        &= \int_{\Gamma_h} (\mathcal{L}_h^{-1/2} \mathcal{P}_h \iota_h^{-1} \mathcal{L}^{1/2} u) (\iota_h^{-1} \overline{\varphi}) \delta_h \, d\sigma_h \\
        &= \int_{\Gamma_h} (\iota_h^{-1} \mathcal{L}^{1/2} u) (\mathcal{L}_h^{-1/2} \mathcal{P}_h (\iota_h^{-1} \overline{\varphi}) \delta_h) \, d\sigma_h \\
        &= \int_{\Gamma} (\mathcal{L}^{1/2} u) (\iota_h \mathcal{L}_h^{-1/2} \mathcal{P}_h (\iota_h^{-1} \overline{\varphi}) \delta_h) (\iota_h \delta_h^{-1}) \, d \sigma \\
        &= \int_{\Gamma} u (\mathcal{L}^{1/2} (\iota_h \mathcal{L}_h^{-1/2} \mathcal{P}_h (\iota_h^{-1} \overline{\varphi}) \delta_h) (\iota_h \delta_h^{-1})) \, d\sigma \\
        &\leq \Vert u \Vert_{L^2(\Gamma)} \Vert \mathcal{L}^{1/2} (\iota_h \mathcal{L}_h^{-1/2} \mathcal{P}_h (\iota_h^{-1} \varphi) \delta_h) (\iota_h \delta_h^{-1}) \Vert_{L^2(\Gamma)}.
    \end{align*}
    We have, using the inequality $\Vert f g \Vert_{H^1(\Gamma)} \leq \Vert f \Vert_{H^1(\Gamma)} (\Vert g \Vert_{L^{\infty}(\Gamma)} + \sum_{i=1}^3 \Vert \underline{\partial}_i g \Vert_{L^{\infty}(\Gamma)})$, Lemma \ref{lemma:sfem}, \ref{lemma:H1-norm-equivalence} and the inequality $\Vert u \Vert_{H^1(\Gamma_h)} \leq C \Vert \mathcal{L}_h^{1/2} u \Vert_{L^2(\Gamma_h)}$ for any $u \in V_h$, 
    \begin{align*}
        \Vert \mathcal{L}^{1/2} (\iota_h \mathcal{L}_h^{-1/2} \mathcal{P}_h (\iota_h^{-1} \varphi) \delta_h) (\iota_h \delta_h^{-1}) \Vert_{L^2(\Gamma)} &\leq C \Vert (\iota_h \mathcal{L}_h^{-1/2} \mathcal{P}_h (\iota_h^{-1} \varphi) \delta_h) (\iota_h \delta_h^{-1}) \Vert_{H^1(\Gamma)} \\
        &\leq C \Vert \iota_h \mathcal{L}_h^{-1/2} \mathcal{P}_h (\iota_h^{-1} \varphi) \delta_h \Vert_{H^1(\Gamma)} \\
        &\qquad \times \bigg( \Vert \iota_h \delta_h^{-1} \Vert_{L^{\infty}(\Gamma)} + \sum_{i=1}^{3} \Vert \underline{\partial}_i \iota_h \delta_h^{-1} \Vert_{L^{\infty}(\Gamma)} \bigg) \\
        &\leq C \Vert \iota_h \mathcal{L}_h^{-1/2} \mathcal{P}_h (\iota_h^{-1} \varphi) \delta_h \Vert_{H^1(\Gamma)} \\
        &\leq C \Vert \mathcal{L}_h^{1/2} \iota_h^{-1} (\iota_h \mathcal{L}_h^{-1/2} \mathcal{P}_h (\iota_h^{-1} \varphi) \delta_h) \Vert_{L^2(\Gamma_h)} \\
        &\leq C \Vert \iota_h^{-1} \varphi \Vert_{L^2(\Gamma_h)} \Vert \delta_h \Vert_{L^{\infty}(\Gamma_h)} \\
        &\leq C \Vert \varphi \Vert_{L^2(\Gamma)}.
    \end{align*}
    This completes the proof. 
\end{proof}

We will need the following two lemmas to show that conditions \ref{cond:abstract-solution-operator}, \ref{cond:abstract-solution-operator-2}, \ref{cond:abstract-solution-operator-3} all hold for the choice of $A$, $A_h$ and $\pi_h$ above. 
\begin{lemma}\label{lemma:ritz-error}
    Let $R_h : V \to \iota_h V_h$ be the Ritz projection associated to the sesquilinear form and inner product $a : V \times V \to \mathbb{C}$, defined as the $a$-orthogonal projection onto $\iota_h V_h$, such that
    \begin{align*}
        a(R_h u, \varphi) = a(u, \varphi) \quad \text{for any} \quad \varphi \in \iota_h V_h. 
    \end{align*}
    Then,
    \begin{align*}
        \Vert (I - R_h) u \Vert_{H^1(\Gamma)} \leq C h \Vert u \Vert_{H^2(\Gamma)}, \quad \Vert (I - R_h) u \Vert_{L^2(\Gamma)} \leq C h^2 \Vert u \Vert_{H^2(\Gamma)},
    \end{align*}
    for any $u \in H^2(\Gamma) \cap V$, while
    \begin{align*}
        \quad \Vert (I - R_h) u \Vert_{L^2(\Gamma)} \leq C h \Vert u \Vert_{H^1(\Gamma)},
    \end{align*}
    for any $u \in V$. 
\end{lemma}
\begin{proof}
    The arguments are the same as those in Section 3.4 in \cite{yagi}. The $H^1(\Gamma)$-error estimate follows by error estimates for the interpolant on the lifted finite element space, $\iota_h V_h$:
    \begin{align*}
        \Vert u - I_h u \Vert_{H^1(\Gamma)} \leq C h \Vert u \Vert_{H^2(\Gamma)}, \quad I_h u := \sum_{n = 1}^{N_h} u(x_n) \iota_h \varphi_n, \quad u \in H^2(\Gamma),
    \end{align*}
    (see e.g. Lemma 5 in \cite{1988-dziuk}), Galerkin-orthogonality, and the continuity and coercivity of the sesquilinear form $a$. The $L^2(\Gamma)$-error estimate follows by the Aubin-Nitche trick, noting that $\mathcal{L}$ satisfies the elliptic regularity estimate \eqref{eq:example-elliptic-regularity-estimate}, and that $a$ is symmetric.
\end{proof}

\begin{lemma}\label{lemma:sfem-error}
    Let $u_h := \mathcal{L}_h^{-1} \mathcal{P}_h \iota_h^{-1} \mathcal{L} u$, and let $R_h$ be the Ritz projection from Lemma \ref{lemma:ritz-error}. Then,
    \begin{align*}
        \Vert R_h u - \iota_h u_h \Vert_{H^1(\Gamma)} \leq C h \Vert u \Vert_{H^2(\Gamma)} \quad \text{and} \quad \Vert R_h u - \iota_h u_h \Vert_{L^2(\Gamma)} \leq C h^2 \Vert u \Vert_{H^2(\Gamma)},
    \end{align*}
    for any $u \in H^2(\Gamma) \cap V$, while
    \begin{align*}
        \Vert R_h u - \iota_h u_h \Vert_{L^2(\Gamma)} \leq C h \Vert u \Vert_{H^1(\Gamma)},
    \end{align*}
    for any $u \in V$. 
\end{lemma}
\begin{proof}
    Note first that by Lemma \ref{lemma:example-negative-rho-condition},  
    \begin{align*}
        \Vert \iota_h u_h \Vert_{H^1(\Gamma)} \leq C \Vert \mathcal{L}_h^{1/2} u_h \Vert_{L^2(\Gamma)} \leq C \Vert \mathcal{L}_h^{-1/2} \mathcal{P}_h \iota_h^{-1} \mathcal{L}^{1/2} \mathcal{L}^{1/2} u \Vert_{L^2(\Gamma)} \leq C \Vert \mathcal{L}^{1/2} u \Vert_{L^2(\Gamma)} \leq C \Vert u \Vert_{H^1(\Gamma)}.
    \end{align*}
    Setting $e_h := R_h u - \iota_h u_h \in \iota_h V_h$, we have using Lemma \ref{lemma:example-bilinear-form-error}
    \begin{align*}
        c \Vert e_h \Vert_{H^1(\Gamma)}^2 &\leq a(e_h, e_h) \\
        &= a(R_h u, e_h) - a(\iota_h u_h, e_h) \pm a_h(u_h, \iota_h^{-1} e_h) \\
        &\leq \vert a(R_h u, e_h) - a_h(u_h, \iota_h^{-1} e_h) \vert + \vert a(\iota_h u_h, e_h) - a_h(u_h, \iota_h^{-1} e_h) \vert \\
        &\leq \vert (\mathcal{L} u, e_h)_{L^2(\Gamma)} - (\iota_h^{-1} \mathcal{L} u, \iota_h^{-1} e_h)_{L^2(\Gamma_h)} \vert + C h^2 \Vert \iota_h u_h \Vert_{H^1(\Gamma)} \Vert e_h \Vert_{H^1(\Gamma)} \\
        &\leq C h^2 \Vert \mathcal{L} u \Vert_{L^2(\Gamma)} \Vert e_h \Vert_{L^2(\Gamma)} + C h^2 \Vert u \Vert_{H^1(\Gamma)} \Vert e_h \Vert_{H^1(\Gamma)}.
    \end{align*}
    Using that $\Vert e_h \Vert_{L^2(\Gamma)} \leq \Vert e_h \Vert_{H^1(\Gamma)}$, and $\Vert \mathcal{L} u \Vert_{L^2(\Gamma)} \leq C \Vert u \Vert_{H^2(\Gamma)}$ gives the first and second inequality of the lemma. 

    For the third, we note by similar arguments as above,
    \begin{align*}
        c \Vert e_h \Vert_{H^1(\Gamma)}^2 &\leq \vert a(R_h u, e_h) - a_h(u_h, \iota_h^{-1} e_h) \vert + \vert a(\iota_h u_h, e_h) - a_h(u_h, \iota_h^{-1} e_h) \vert \\
        &\leq \vert (\mathcal{L} u, e_h)_{L^2(\Gamma)} - (\iota_h^{-1} \mathcal{L} u, \iota_h^{-1} e_h)_{L^2(\Gamma_h)} \vert + C h^2 \Vert u \Vert_{H^1(\Gamma)} \Vert e_h \Vert_{H^1(\Gamma)}.
    \end{align*}
    Note that, by Lemma \ref{lemma:sfem}, 
    \begin{align*}
        \vert (\mathcal{L} u, e_h)_{L^2(\Gamma)} - (\iota_h^{-1} \mathcal{L} u, \iota_h^{-1} e_h)_{L^2(\Gamma_h)} \vert &\leq \vert (\mathcal{L} u, e_h)_{L^2(\Gamma)} - (\mathcal{L} u, e_h \iota_h \delta_h^{-1})_{L^2(\Gamma)} \vert \\
        &= (\mathcal{L} u, e_h (1 - \iota_h \delta_h^{-1}))_{L^2(\Gamma)} \\
        &= (\mathcal{L}^{1/2} u, \mathcal{L}^{1/2} e_h (1 - \iota_h \delta_h^{-1}))_{L^2(\Gamma)} \\
        &\leq \Vert u \Vert_{H^1(\Gamma)} \Vert e_h \Vert_{H^1(\Gamma)} \\
        &\qquad \times \bigg( \Vert 1 - \iota_h \delta_h^{-1} \Vert_{L^{\infty}(\Gamma)} + \sum_{i=1}^{3} \Vert \underline{\partial}_i (1 - \iota_h \delta_h^{-1}) \Vert_{L^{\infty}(\Gamma)} \bigg) \\
        &\leq C h \Vert u \Vert_{H^1(\Gamma)} \Vert e_h \Vert_{H^1(\Gamma)}. 
    \end{align*}
    By combining the terms above, we get the third inequality of the lemma. 
\end{proof}

\begin{lemma}
    Let $A, A_h$ and $\pi_h$ be as above. Then the following estimate holds,
    \begin{align*}
        \Vert (I + A)^{\alpha} ((I + A)^{-1} - (I + A_h)^{-1} \pi_h) x \Vert_{L^2(\Gamma)} \leq C h^{2-2\alpha} \Vert x \Vert_{L^2(\Gamma)},
    \end{align*}
    for $\alpha \in \{ 0, 1/2 \}$, while,
    \begin{align*}
        \Vert ((I + A)^{-1} - (I + A_h)^{-1} \pi_h) (I + A)^{1/2} x \Vert_{L^2(\Gamma)} \leq C h \Vert x \Vert_{L^2(\Gamma)},
    \end{align*}
    and so condition \ref{cond:abstract-solution-operator}, \ref{cond:abstract-solution-operator-2} and \ref{cond:abstract-solution-operator-3} holds for this choice of $A$, $A_h$ and $\pi_h$.  
\end{lemma}
\begin{proof}
    This follows by decomposing the error, $(I + A)^{-1} x - (I + A_h)^{-1} \pi_h x = (I + A)^{-1} x - R_h (I + A)^{-1} x - (I + A_h)^{-1} \pi_h x + R_h (I + A)^{-1} x$, and then using Lemma \ref{lemma:ritz-error} and \ref{lemma:sfem-error}. 
\end{proof}

Denote the semigroup generated by $-L$ on $H$ by $S(\cdot)$. As a consequence of the observations above, Theorem \ref{theorem:1}, \ref{theorem:fully-discrete-semigroup-approximation-2}, \ref{theorem:fully-discrete-semigroup-approximation-4} and Lemma \ref{lemma:negative-rho}, \ref{lemma:1}, \ref{lemma:2}, we have the following Corollaries.
\begin{corollary}[Semidiscrete error estimates]\label{cor:semidiscrete-numerical-example}
    Let $L_h$ be as above, and let $S_h$ be the semigroup generated by $-L_h$. Then, the following estimate holds,
    \begin{align*}
        \Vert (S(t) - \iota_h S_h(t) \mathcal{P}_h \iota_h^{-1}) u_0 \Vert_H \leq C e^t t^{-\theta / 2 + \rho / 2} h^{\theta} \Vert \mathcal{L}^{\rho / 2} u_0 \Vert_H, \quad \text{for any } u_0 \in D(\mathcal{L}^{\rho / 2}),
    \end{align*}
    where $\theta \in [0,2]$, $\rho \in [-1, \theta] \cap [-2 + \theta, \theta]$, while
    \begin{align*}
        \Vert \mathcal{L}^{\alpha} (S(t) - \iota_h S_h(t) \mathcal{P}_h \iota_h^{-1}) u_0 \Vert_H \leq C e^t t^{-\theta / 2} h^{\theta - 2 \alpha} \Vert u_0 \Vert_H, \quad \text{for any } u_0 \in H,
    \end{align*}
    where $\alpha \in [-1/2, 1/2]$, $\theta \in [2\alpha,2] \cap [0,2 + 2\alpha]$, and $C > 0$.
\end{corollary}
\begin{proof}
    First, note that $e^{-A_h t} \pi_h = \iota_h S_h(t) \mathcal{P}_h \iota_h^{-1}$, where $e^{-A_h t}$ is the semigroup generated by $-A_h$ on $X_h$. Since all conditions of Assumption \ref{assumption:abstract-model}, and \ref{cond:discrete-growth}, \ref{cond:abstract-solution-operator-2}, \ref{cond:negative-rho} and \ref{cond:abstract-solution-operator-3} holds, the first inequality of this corollary is a consequence of Theorem \ref{theorem:1} and Lemma \ref{lemma:negative-rho}. The second inequality follows by Lemma \ref{lemma:1} and \ref{lemma:2}.
\end{proof}
\begin{corollary}[Fully discrete error estimates]\label{cor:fully-discrete-numerical-example}
    Let $L_h$ be as above, $r(z) = (1 + z)^{-1}$, and $S_{h,\Delta t}$ be given as in \eqref{eq:fully-discrete-semigroup-operator} where $L_h$ is used. Then, the following estimate holds,
    \begin{align*}
        \Vert (S(t) - \iota_h S_{h,\Delta t}(t) \mathcal{P}_h \iota_h^{-1} ) u_0 \Vert_H \leq C e^{ct} t^{-\theta / 2 + \rho / 2} h^{\theta} \Vert \mathcal{L}^{\rho / 2} u_0 \Vert_H, \quad \text{for any } u_0 \in D(\mathcal{L}^{\rho / 2}),
    \end{align*}
    where $\theta \in [0,2]$, $\rho \in [-1, \theta] \cap [-2 + \theta, \theta]$, and $C, c > 0$. 
\end{corollary}
\begin{proof}
    Arguing similarly as in the proof of Corollary \ref{cor:semidiscrete-numerical-example}, this corollary is a consequence of Theorem \ref{theorem:fully-discrete-semigroup-approximation-2} and \ref{theorem:fully-discrete-semigroup-approximation-4}.
\end{proof}

\subsection*{Acknowledgements}
The author is grateful for comments on parts of this manuscripts by Annika Lang, Geir-Arne Fuglstad and Espen R. Jakobsen. 

\printbibliography

\appendix
\section{Proofs of analytic semigroup properties}\label{app:prelim-proofs}

\begin{proof}[Proof of Lemma \ref{lemma:analytic-semigroup}]
    For the first four statements, see for example Theorem 6.8 and 6.13 in \cite{pazy}. 

    The inequality \eqref{eq:decay-no-lambda} is not contained in that theorem, since $A$ is not strictly positive, but the proof is essentially the same: since $\frac{d}{dt} S(t) x = -A S(t) x$, and $A$ and $S(t)$ commute we have,
    \begin{align*}
        \Vert (I - S(t)) x \Vert_X = \Vert \int_0^t A S(s) x \, ds \Vert_X \leq \int_0^t \Vert S(s) \Vert_{L(X)} \Vert A x \Vert_X \, ds \leq C e^{\lambda t} t \Vert A x \Vert_X.
    \end{align*}
    By the boundedness of $S(t)$, we also find that $\Vert (I - S(t)) x \Vert_X \leq C e^{\lambda t}$. Further, using that $A(\lambda + A)^{-1} = I - \lambda (\lambda + A)^{-1}$, we see that,
    \begin{align*}
        \Vert A (\lambda + A)^{-1} \Vert_{L(X)} \leq 1 + \lambda \Vert (\lambda + A)^{-1} \Vert_{L(X)} \leq C,
    \end{align*}
    so that,
    \begin{align*}
        \Vert A x \Vert_X = \Vert A (\lambda + A)^{-1} (\lambda + A) x \Vert_X \leq \Vert A (\lambda + A)^{-1} \Vert_{L(X)} \Vert (\lambda + A) x \Vert_X \leq C \Vert (\lambda + A) x \Vert_X.
    \end{align*}
    Combining these observations, we have for $\alpha \in \{ 0, 1 \}$,
    \begin{align*}
        \Vert (I - S(t)) (\lambda + A)^{-\alpha} \Vert_{L(X)} \leq C e^{\lambda t} t^{\alpha},
    \end{align*}
    and by Lemma \ref{lemma:interpolation-inequality}, it holds for $\alpha \in [0,1]$.
\end{proof}

The following Lemma is key for obtaining the (uniform in $h$) exponential decay of the semigroup of $-A_h$. This will be used frequently in the rest of the proofs in this appendix. 
\begin{lemma}\label{lemma:shifted-discrete}
    Let $A_h$ be as in Assumption \ref{assumption:abstract-model}. Then there is some $\epsilon > 0$, such that $-\epsilon + A_h$ satisfies condition \ref{cond:discrete-sectorial} for some $M', \delta'$, possibly larger than $M$ and $\delta$.
\end{lemma}
\begin{proof}
    Assume without loss of generality that we may choose $\lambda = 0$. By similar arguments as those in Chapter 2 of \cite{yagi}, it suffices to show that for any $z \in \Sigma_{\pi / 2}^c \setminus \{ 0 \}$ (with the superscript $c$ denoting complement), we can find $M'' > 0$ such that
    \begin{align*}
        \Vert (z + \epsilon - A_h)^{-1} \Vert_{L(X_h)} \leq \frac{M''}{\vert z \vert}.
    \end{align*}
    Note first that by \ref{cond:abstract-solution-operator} we have $\Vert A_h^{-1} \Vert_{L(X_h)} \leq \Vert A^{-1} \Vert_{L(X)} + C =: c$. By expanding $(z - A_h)^{-1}$ around $0$, so that
    \begin{align*}
        (z - A_h)^{-1} = \sum_{n = 0}^{\infty} (-1)^{n} z^n A_h^{-n-1}, \quad \vert z \vert < \Vert A_h^{-1} \Vert_{L(X_h)}^{-1},
    \end{align*}
    we find that for $\vert z \vert \leq c^{-1} / 2$ we have $\Vert (z - A_h)^{-1} \Vert_{L(X_h)} \leq 2 c$. Therefore, we can find an $0 < \epsilon \leq \cos(\delta) c^{-1} / 4$, such that the resolvent of $-\epsilon + A_h$ contains $\Sigma_{\pi / 2}^c$, and,
    \begin{enumerate}
        \item $\Vert (z + \epsilon - A_h)^{-1} \Vert_{L(X_h)} \leq 2c$, whenever $\vert z \vert \leq \cos(\delta)^{-1} \epsilon$, while \smallskip
        \item $\{ \vert z \vert \geq \cos(\delta)^{-1} \epsilon \} \cap \epsilon + \Sigma_{\pi / 2}^c \subseteq \Sigma_{\delta}^c$.
    \end{enumerate}
    By \ref{cond:discrete-sectorial} we therefore have for any $z \in \Sigma_{\pi / 2}^c$,
    \begin{align*}
        \Vert (z + \epsilon - A_h)^{-1} \Vert_{L(X_h)} \leq 
        \begin{cases}
            2c, \quad &\vert z \vert \leq \cos(\delta)^{-1} \epsilon, \\
            \frac{M}{\vert z \vert - \epsilon}, \quad &\vert z \vert \geq \cos(\delta)^{-1} \epsilon.
        \end{cases}
    \end{align*}
    Thus we may choose $M'' = \max(2 c \cos(\delta)^{-1} \epsilon, (1-\cos(\delta) )M)$. 
\end{proof}

\begin{proof}[Proof of Lemma \ref{lemma:discrete-semigroup-smoothing}]
    We first show that $\Vert S_h(t) \Vert_{L(X_h)} \leq C e^{(\lambda - \epsilon) t}$. By the definition of $S_h$, we have,
    \begin{align*}
        S_h(t) = \frac{1}{2\pi i} \int_{-\lambda + \epsilon + \gamma} e^{-zt} (z - A_h)^{-1} \, dz, 
    \end{align*}
    where $\gamma = \gamma_R \cup \gamma_R^{\infty}$, where $\gamma_R = \{ R e^{i \phi}, \ \vert \phi \vert \in [\delta', \pi] \}$ and $\gamma_R^{\infty} := \{ s e^{\pm i\delta'}, \ s \geq R \}$, oriented counter clockwise, with $\delta'$ as in Lemma \ref{lemma:shifted-discrete}. Therefore, 
    \begin{align*}
        \Vert S_h(t) \Vert_{L(X_h)} &= \frac{1}{2\pi} \Vert \int_{\gamma} e^{(\lambda-\epsilon) t - z t} (z + \epsilon - \lambda - A_h)^{-1} \, dz \Vert_{L(X_h)} \\
        &\leq \frac{1}{2\pi} e^{(\lambda-\epsilon) t} \bigg( \int_{\gamma_R} e^{-\mathrm{Re}(z) t} \Vert (z + \epsilon - \lambda - A_h)^{-1}\Vert_{L(X_h)} \, \vert dz \vert \\
        &\quad + \int_{\gamma_R^{\infty}} e^{-\mathrm{Re}(z) t} \Vert (z + \epsilon - \lambda - A_h)^{-1}\Vert_{L(X_h)} \, \vert dz \vert \bigg).
    \end{align*}
    Using condition Lemma \ref{lemma:shifted-discrete} we find 
    \begin{align*}
        \int_{\gamma_R} e^{-\mathrm{Re}(z) t} \Vert (z + \epsilon - \lambda - A_h)^{-1}\Vert_{L(X_h)} \, \vert dz \vert \leq 2 \pi M' e^{Rt},
    \end{align*}
    and,
    \begin{align*}
        \int_{\gamma_R^{\infty}} e^{-\mathrm{Re}(z) t} \Vert (z + \epsilon - \lambda + A_h)^{-1}\Vert_{L(X_h)} \, \vert dz \vert &\leq C \int_R^{\infty} e^{- \cos(\delta') s t} s^{-1} \, ds \\
        &\leq C R^{-1} t^{-1} M' e^{R t}. 
    \end{align*}
    Therefore, setting $R = t^{-1}$, we get that $\Vert S_h(t) \Vert_{L(X_h)} \leq C e^{(\lambda - \epsilon) t}$, and the inequality of this lemma with $\alpha = 0$ follows by the (uniform in $h$) boundedness of $\pi_h$. 
    
    The proof of the inequality with $\alpha \geq 0$ is similar to showing \eqref{eq:analytic-semigroup-property-1-2}. It suffices to consider the case $\alpha \geq 1$ integer, since the cases $\alpha \geq 0$ real follows by the first interpolation inequality of Lemma \ref{lemma:interpolation-inequality}, noting that this inequality also applies to the operators $A_h$. More precisely, we may apply it to the expression
    \begin{align*}
        (\lambda + A_h)^{\alpha} S_h(t) = (\lambda + A_h)^{\alpha - n} (\lambda + A_h)^n S_h(t), \quad \alpha \in [n,n+1],
    \end{align*}
    for $n$ integer. 
    
    Take $\lambda$ as in condition \ref{cond:sectorial}, $\gamma$ as above with $R = 0$, and note that for some $n \geq 1$ integer,
    \begin{align*}
        (\lambda + A_h)^n S_h(t) = \frac{1}{2\pi i} \int_{-\lambda + \epsilon + \gamma} (\lambda + z)^n e^{-zt} (z - A_h)^{-1} \, dz = \frac{1}{2\pi i} \int_{\gamma} (z + \epsilon)^n e^{(\lambda - \epsilon)t - zt} (z + \epsilon - \lambda - A_h)^{-1} \, dz.
    \end{align*}
    We therefore gather that,
    \begin{align*}
        \Vert (\lambda + A_h)^n S_h(t) \Vert_{L(X_h)} &\leq C e^{(\lambda - \epsilon) t} \int_{\gamma} \vert z \vert^n e^{-\cos(\delta') \vert z \vert t} \frac{M'}{\vert z \vert} \, \vert d z \vert \\ 
        &\leq C e^{(\lambda - \epsilon) t} \bigg( \sup_{\vert z \vert \geq 0} \vert z \vert^{n - 1} e^{-\cos(\delta') \vert z \vert t / 2} \bigg) \int_{\gamma} e^{-\cos(\delta') \vert z \vert t / 2} \, \vert d z \vert \\
        &\leq C e^{(\lambda-\epsilon) t} t^{-n + 1} \int_0^{\infty} e^{-\cos(\delta') s t / 2} \, ds \\
        &\leq C e^{(\lambda-\epsilon) t} t^{-n}.
    \end{align*}
    This finishes the proof. 
\end{proof}
\section{Proofs related to rational approximation of analytic semigroups}\label{app:rational-approximation-proofs}

\begin{proof}[Proof of Lemma \ref{lemma:fully-discrete-inequalities-1}]
    To show this we first note that with $\epsilon, \delta'$ and $M'$ as in Lemma \ref{lemma:shifted-discrete}, the shifted and scaled operator $\Delta t (\lambda - \epsilon + A_h)$, satisfies:
    \begin{align}\label{eq:shifted-and-scaled-sectorial}
        \sigma(\Delta t(\lambda - \epsilon + A_h)) \subseteq \Sigma_{\delta'}, \quad \text{ and } \quad \Vert (z - \Delta t(\lambda - \epsilon + A_h))^{-1} \Vert_{L(X_h)} \leq \frac{M'}{\vert z \vert}, \quad z \notin \Sigma_{\delta'}. 
    \end{align}
    Both the inclusion and the inequality above follows by noting that, 
    \begin{align*}
        \Vert (z - \Delta t(\lambda - \epsilon + A_h))^{-1} \Vert_{L(X_h)} = \Delta t^{-1} \Vert (\Delta t^{-1} z - \lambda + \epsilon - A_h)^{-1} \Vert_{L(X_h)},
    \end{align*}
    and then applying the estimate $\Vert (z - \lambda + \epsilon - A_h)^{-1} \Vert_{L(X_h)} \leq M' \vert z \vert^{-1}$ for any $z \notin \Sigma_{\delta'}$. 

    Owing to Lemma \ref{lemma:discrete-semigroup-smoothing}, it suffices to show that $\Vert r(\Delta t A_h)^n \Vert_{L(X_h)} \leq C e^{c (\lambda - \epsilon) t_n}$. This holds when $n = 0$, while for $n = 1$,
    \begin{align*}
        \Vert r(\Delta t A_h) \Vert_{L(X_h)} = \Vert (1 - \lambda \Delta t + \Delta t (\lambda + A_h))^{-1} \Vert_{L(X_h)} \leq \frac{M}{1-\lambda \Delta t} \leq 2 M, \quad \text{ for all } \Delta t < \frac{1}{2\lambda}. 
    \end{align*}
    
    For $n > 1$ we have (see Chapter VII.9 \cite{dunford}),
    \begin{align*}
        r(\Delta t A_h)^n = \frac{1}{2\pi i} \int_{\gamma} r(z - (\lambda - \epsilon) \Delta t)^n (z - \Delta t (\lambda - \epsilon + A_h))^{-1} \, dz,
    \end{align*}
    where $\gamma = \gamma_R \cup \gamma_R^{\infty}$, where $\gamma_R = \{ R e^{i \theta}, \ \vert \theta \vert \in [\delta', \pi] \}$, for any $0 < R < 1 - (\lambda-\epsilon) \Delta t$, and $\gamma_R^{\infty} := \{ s e^{\pm i\delta'}, \ s \geq R \}$, oriented counter clockwise. The condition on $R$ ensures that the pole $z = -1 + (\lambda-\epsilon) \Delta t$ of $r(\cdot - (\lambda - \epsilon) \Delta t)$ is not contained in our contour. Therefore, 
    \begin{align*}
        \Vert r(\Delta t A_h)^n \Vert_{L(X_h)} &\leq \frac{1}{2\pi} ( \int_{\gamma_R} + \int_{\gamma_R^{\infty}} ) \vert r(z - (\lambda-\epsilon) \Delta t) \vert^n \Vert (z-\Delta t(\lambda - \epsilon + A_h))^{-1} \Vert_{L(X_h)} \, \vert dz \vert. 
    \end{align*}
    Since $n > 1$, we have that for $N$ large enough, $1-(\lambda - \epsilon) \Delta t - n^{-1} > 0$ for all $n = 2, \dots, N$. Setting $R = n^{-1}$, we have by \eqref{eq:shifted-and-scaled-sectorial}
    \begin{align*}
        &\int_{\gamma_R} \vert r(z-(\lambda-\epsilon) \Delta t) \vert^n \Vert (z - \Delta t(\lambda - \epsilon + A_h))^{-1} \Vert_{L(X_h)} \\
        &\qquad \leq 2 \pi R \sup_{\theta \in [0,2\pi]} \vert 1 + e^{i \theta} R - (\lambda-\epsilon) \Delta t \vert^{-n} \frac{M'}{R} \leq 2 \pi M' (1 - n^{-1} - (\lambda-\epsilon) \Delta t)^{-n}.
    \end{align*}
    It follows that this term is bounded by $C e^{c (\lambda-\epsilon) t_n}$, for some $c,C > 0$ independent of $n$ and $\Delta t$\footnote{For $n > 1$, $x \in [0,1/2)$, $(1-x-n^{-1})^{-1} \leq (1-n^{-1})^{-1} (1 - 2x)^{-1}$, while for $x \in (-1,0]$, $(1-x-n^{-1}) \leq (1-n^{-1})^{-1}(1-x)^{-1}$. At the same time, $(1-n^{-1})^{-n} \leq e^2$, $(1-x)^{-1} \leq e^{2x}$ for $x \in [0,1/2]$, and $(1-x)^{-1} \leq e^{x/2}$ for $x \in [-1/2,0]$.}.

    At the same time, still with $R = n^{-1}$, and now for $N$ large enough (that $1 + \cos(\delta') n^{-1} - (\lambda-\epsilon) \Delta t > 0$ for all $n > 1$),
    \begin{align*}
        \int_{\gamma_R^{\infty}} \vert r(z-(\lambda-\epsilon) \Delta t) \vert^n \Vert (z - \Delta t(\lambda - \epsilon + A_h))^{-1} \Vert_{L(X_h)} &\leq C \int_R^{\infty} \vert 1 + e^{i \delta'} s - (\lambda-\epsilon) \Delta t \vert^{-n} \frac{M'}{s} \, ds \\ 
        &\leq C \int_R^{\infty} \vert 1 + \cos(\delta') s - (\lambda-\epsilon) \Delta t \vert^{-n} \frac{M'}{s} \, ds \\
        &\leq C (n-1)^{-1} (1 + \cos(\delta') R - (\lambda-\epsilon) \Delta t)^{-n + 1} R^{-1} \\
        &\leq C (1 + \cos(\delta') n^{-1} - (\lambda-\epsilon) \Delta t)^{-n},
    \end{align*}
    as $n / (n-1) \leq 2$ since $n > 1$. By the same arguments as for the previous term, this term is bounded by $C e^{c (\lambda-\epsilon) t_n}$, for some $C > 0$ independent of $n$ and $\Delta t$.
\end{proof}

The following lemma is key for proving Lemma \ref{lemma:fully-discrete-inequalities-2}, \ref{lemma:fully-discrete-inequalities-3} and \ref{lemma:fully-discrete-inequalities-4}. 
\begin{lemma}\label{lemma:rational-approximation-errors}
    Let $F_n(z) := r(z)^n - e^{-z n}$, for some $z \in \mathbb{C}$, $n \in \mathbb{N}_0$, where $r(z) = (1 + z)^{-1}$. Then for some $\omega \in [-1/2, 1/2]$, there is $C_R, c_R, C, c, c' > 0$ such that
    \begin{align}\label{eq:inequality-along-gamma-0-R}
        \vert F_n(z + \omega) \vert \leq C_R \vert z + \omega \vert^2 n e^{-c_R \mathrm{Re}(z) - c n \omega}, \quad z \in \gamma_0^R, 
    \end{align}
    and
    \begin{align}\label{eq:inequality-along-R-infty}
        \vert F_n(z + \omega) \vert \leq C e^{-c n \omega} ( (1 + c' \mathrm{Re}(z))^{-n} + e^{-n \mathrm{Re}(z)} ), \quad z \in \gamma_R^{\infty},
    \end{align}
    while,
    \begin{align}\label{eq:inequality-along-gamma-r}
        \vert F_n(z + \omega) \vert \leq (1 - \vert z + \omega \vert)^{-1} \vert z + \omega \vert^2 n e^{c n \vert z \vert - c' n \omega}, \quad \vert z + \omega \vert < 1,
    \end{align}
    where $\gamma_0^R := \{ s e^{i \delta'}, s \in [0, R] \}$, $\gamma_R^{\infty} := \{ s e^{i \delta'}, s \geq R \}$ and $\delta'$ is as in Lemma \ref{lemma:shifted-discrete}. 
\end{lemma}
\begin{proof}
    Note first that for some $\omega \in [-1/2, 1/2]$ there is $C_R > 0$ such that,
    \begin{align*}
        F_1(z) \leq C_R \vert z + \omega \vert^2, \quad z \in \gamma_0^R.
    \end{align*}
    Secondly, note that we can find $c_R > 0$ such that,
    \begin{align*}
        \vert r(z + \omega) \vert \leq \vert 1 + \omega \vert^{-1} \vert 1 + \frac{\mathrm{Re}(z)}{1+\omega} \vert^{-1} \leq e^{-c_R \mathrm{Re}(z)} 
        \begin{cases}
            e^{-2 \omega}, \quad &\omega < 0, \\
            e^{-\omega/2}, \quad &\omega \geq 0, 
        \end{cases} \quad z \in \gamma_0^R.
    \end{align*}
    Finally, combining these two inequalities with the identity $\varphi^n - \psi^n = (\varphi - \psi) \sum_{j = 0}^{n-1} \varphi^j \psi^{n-1-j}$, so that,
    \begin{align}\label{eq:geometric-series}
        F_n(z + \omega) = (r(z + \omega) - e^{-z-\omega}) \sum_{j=0}^{n-1} r(z + \omega)^j e^{-(n-1-j) (z+\omega)},
    \end{align}
    we get the first inequality of this lemma (where $C_R, c_R, c > 0$ are possibly different from those above). The second inequality follows by similar observations as the first, while the third follows by combining the inequalities
    \begin{align*}
        F_1(z + \omega) \leq (1-\vert z + \omega \vert)^{-1} \vert z + \omega \vert^2, \quad \vert r(z + \omega) \vert \leq e^{c\vert z \vert - c' \omega}, \quad \vert z + \omega \vert < 1,
    \end{align*}
    with \eqref{eq:geometric-series}. 
\end{proof}

\begin{proof}[Proof of Lemma \ref{lemma:fully-discrete-inequalities-2}]
    Let $F_n$ be as in Lemma \ref{lemma:rational-approximation-errors}, and let $M', \delta'$ be as in Lemma \ref{lemma:shifted-discrete}. The inequality of this lemma holds if we are able to show that,
    \begin{align*}
        \Vert F_n(\Delta t A_h) (\lambda + A_h)^{-1} \Vert_{L(X_h)} \leq C e^{c (\lambda-\epsilon) t_n} \Delta t,
    \end{align*}
    for some $C, c > 0$. Note that,
    \begin{align*}
        \Vert F_n(\Delta t A_h) (\lambda + A_h)^{-1} \Vert_{L(X_h)}
        &\leq \Vert F_n(\Delta t A_h) (2 \lambda + A_h)^{-1} \Vert_{L(X_h)} \Vert (2 \lambda + A_h) (\lambda + A_h)^{-1} \Vert_{L(X_h)} \\
        &\leq C \Vert F_n(\Delta t A_h) (2 \lambda + A_h)^{-1} \Vert_{L(X_h)}.
    \end{align*}
    Therefore, it suffices to show that, 
    \begin{align*}
        \Vert F_n(\Delta t A_h) (2 \lambda + A_h)^{-1} \Vert_{L(X_h)} \leq C e^{c (\lambda-\epsilon) t_n} \Delta t. 
    \end{align*}
    The reason we are considering the composition with $(2\lambda + A_h)^{-1}$ as opposed to $(\lambda + A_h)^{-1}$ is because it is then easier to find a suitable path of integration when using the residue theorem. Arguing as in the proof of Lemma \ref{lemma:fully-discrete-inequalities-1}, we have for $\gamma = \gamma_r \cup \gamma_r^R \cup \gamma_R^{\infty} = \{ r e^{i \theta}, \vert \theta \vert \in [\delta', \pi] \} \cup \{ s e^{\pm i\delta'}, \ s \in [r, R] \} \cup \{ s e^{\pm i \delta'}, \ s \geq R \}$, $0 < r < (\lambda+\epsilon) \Delta t$, 
    \begin{align*}
        \Delta t^{-1} F_n(\Delta t A_h) (2 \lambda + A_h)^{-1} = \frac{1}{2 \pi i} \int_{\gamma} F_n(z - (\lambda-\epsilon) \Delta t) (z + (\lambda+\epsilon) \Delta t)^{-1} (z - \Delta t(\lambda - \epsilon + A_h))^{-1} \, dz.
    \end{align*}
    Therefore, it suffices to show that,
    \begin{align*}
        &\Vert \frac{1}{2 \pi i} \int_{\gamma} F_n(z - (\lambda-\epsilon) \Delta t) (z + (\lambda+\epsilon) \Delta t)^{-1} (z - \Delta t(\lambda - \epsilon + A_h))^{-1} \, dz \Vert_{L(X_h)} \\
        &\qquad \leq \frac{1}{2\pi} ( \int_{\gamma_r} + \int_{\gamma_r^R} + \int_{\gamma_{R}^{\infty}} ) \vert F_n(z-(\lambda-\epsilon) \Delta t) \vert \vert z + (\lambda+\epsilon) \Delta t \vert^{-1} \Vert (z - \Delta t (\lambda - \epsilon + A_h))^{-1} \Vert_{L(X_h)} \, \vert dz \vert \\
        &\qquad \leq C e^{c (\lambda-\epsilon) t_n}.
    \end{align*}
    
    Using \eqref{eq:inequality-along-gamma-r} for the integral along $\gamma_r$, we therefore have, 
    \begin{align*}
        &\int_{\gamma_r} \vert F_n(z - (\lambda-\epsilon) \Delta t) \vert \vert z + (\lambda+\epsilon) \Delta t \vert^{-1} \Vert (z - \Delta t (\lambda - \epsilon + A_h))^{-1} \Vert_{L(X_h)} \, \vert dz \vert \\
        &\qquad \leq C \int_{\gamma_r} \bigg( n (1 - r - \lambda \Delta t)^{-1}(r^2 + (\lambda-\epsilon)^2 \Delta t^2) e^{c n r + c' (\lambda - \epsilon) t_n} \bigg) \bigg( ((\lambda+\epsilon) \Delta t - r)^{-1} \bigg) M' r^{-1} \, \vert dz \vert, 
    \end{align*}
    and setting $r = \lambda \Delta t$, we have for $N$ large that this is bounded by $C e^{c \lambda t_n - c' \epsilon t_n}$ for some $C, c, c' > 0$ (and we may redefine $\epsilon > 0$ to get an exponent of the form $c(\lambda - \epsilon)$). For the integral along $\gamma_r^R$, we set $R = 1$ and use \eqref{eq:inequality-along-gamma-0-R} to get,
    \begin{align*}
        &C \int_{\gamma_r^R} \bigg( e^{-c n \mathrm{Re}(z) + c (\lambda - \epsilon) t_n} n (\vert z \vert^2 + (\lambda-\epsilon)^2 \Delta t^2) \bigg) \bigg( \vert z + (\lambda+\epsilon) \Delta t \vert^{-1} \bigg) M' \vert z \vert^{-1} \, \vert dz \vert \\
        &\qquad \leq C e^{c(\lambda - \epsilon) t_n} \int_r^R \bigg( e^{-c n s} n (\cos(\delta')^{-2} s^2 + (\lambda-\epsilon)^2 \Delta t^2 ) \bigg) \bigg( (s + \lambda \Delta t)^{-1} \bigg) M' s^{-1} \, ds \\
        &\qquad \leq C e^{c (\lambda-\epsilon) t_n} \int_{\lambda t_n}^{\infty} e^{-c y} (\cos(\delta')^{-2} y^2 + \lambda^2 t_n^2 ) (y + \lambda t_n)^{-1} y^{-1} \, dy \\
        &\qquad \leq C e^{c (\lambda-\epsilon) t_n} e^{c' \lambda t_n},
    \end{align*}
    where we have set $y = s n$. For the final integral along $\gamma_{R}^{\infty}$, we have by \eqref{eq:inequality-along-R-infty}, 
    \begin{align*}
        &\int_{\gamma_R^{\infty}} \vert F_n(z-(\lambda-\epsilon) \Delta t) \vert \vert z + (\lambda+\epsilon) \Delta t \vert^{-1} \Vert (z - \Delta t (\lambda - \epsilon + A_h))^{-1} \Vert_{L(X_h)} \, \vert dz \vert \\
        &\qquad \leq C e^{c(\lambda - \epsilon) t_n} \int_{\gamma_{R}^{\infty}} \vert z + (\lambda+\epsilon) \Delta t \vert^{-1} \vert z \vert^{-1} \, \vert dz \vert \\
        &\qquad \leq C e^{c(\lambda-\epsilon) t_n} \int_1^{\infty} s^{-2} \, ds. 
    \end{align*}
    This finishes the proof. 
\end{proof}

\begin{proof}[Proof of Lemma \ref{lemma:fully-discrete-inequalities-3}]
    Let $F_n(z)$ be as in Lemma \ref{lemma:rational-approximation-errors}, and let $\epsilon, \delta'$ and $M'$ be as in Lemma \ref{lemma:shifted-discrete}. It suffices to show that, 
    \begin{align*}
        \Vert F_n(\Delta t A_h) \Vert_{L(X_h)} \leq C e^{c(\lambda-\epsilon) t_n} t_n^{-1} \Delta t. 
    \end{align*}
    Arguing as in the proof of Lemma \ref{lemma:fully-discrete-inequalities-1}, we have, 
    \begin{align*}
        F_n(\Delta t A_h) = \frac{1}{2\pi i} \int_{\gamma} F_n(z - (\lambda-\epsilon) \Delta t) (z - \Delta t (\lambda - \epsilon + A_h))^{-1} \, dz,
    \end{align*}
    where $\gamma = \gamma_r \cup \gamma_r^R \cup \gamma_R^{\infty}$ is as in the proof of Lemma \ref{lemma:fully-discrete-inequalities-2}. Therefore, we need to show that, 
    \begin{align*}
        &\Vert \frac{1}{2 \pi i} \int_{\gamma} F_n(z - (\lambda-\epsilon) \Delta t) (z - \Delta t(\lambda - \epsilon + A_h))^{-1} \, dz \Vert_{L(X_h)} \\
        &\qquad \leq \frac{1}{2\pi} ( \int_{\gamma_r} + \int_{\gamma_r^R} + \int_{\gamma_R^{\infty}} ) \vert F_n(z-(\lambda-\epsilon) \Delta t) \vert \Vert (z - \Delta t (\lambda - \epsilon + A_h))^{-1} \Vert_{L(X_h)} \, \vert dz \vert \\
        &\qquad \leq C e^{c(\lambda-\epsilon) t_n} t_n^{-1} \Delta t.
    \end{align*}
    For $n = 1$, this inequality holds by Lemma \ref{lemma:fully-discrete-inequalities-1}. For $n > 1$ we have by using \eqref{eq:inequality-along-gamma-r} for the integral along $\gamma_r$, and setting $r = n^{-1} < 1 - (\lambda-\epsilon) \Delta t$,
    \begin{align*}
        &\int_{\gamma_r} \vert F_n(z-(\lambda-\epsilon) \Delta t) \vert \Vert (z - \Delta t (\lambda - \epsilon + A_h))^{-1} \Vert_{L(X_h)} \, \vert dz \vert \\
        &\qquad \leq C e^{c(\lambda - \epsilon) t_n}\int_{\gamma_r} \bigg( e^{c' n r} n (r^2 + (\lambda-\epsilon)^2 \Delta t^2) \bigg) M' r^{-1} \vert dz \vert \\
        &\qquad \leq C e^{c(\lambda-\epsilon) t_n} n^{-1} (1 + \lambda^2 t_n^2) \\
        &\qquad \leq C e^{c(\lambda-\epsilon) t_n} t_n^{-1} \Delta t,
    \end{align*}
    since $n^{-1} = t_n^{-1} \Delta t$. For the next integral, we set $R = 1$ and compute using \eqref{eq:inequality-along-gamma-0-R}, 
    \begin{align*}
        &\int_{\gamma_r^R} \vert F_n(z-(\lambda-\epsilon) \Delta t) \vert \Vert (z - \Delta t (\lambda - \epsilon + A_h))^{-1} \Vert_{L(X_h)} \, \vert dz \vert \\
        &\qquad \leq e^{c(\lambda-\epsilon)t_n} \int_{\gamma_r^R} \bigg( e^{-c' n \mathrm{Re}(z)} n (\vert z \vert^2 + (\lambda-\epsilon)^2 \Delta t^2) \bigg) M' \vert z \vert^{-1} \, \vert dz \vert \\
        &\qquad \leq e^{c(\lambda - \epsilon)t_n} \int_r^R n e^{-c' n s} (\cos(\delta')^{-2} s^2 + (\lambda-\epsilon)^2 \Delta t^2) M' s^{-1} \, ds \\
        &\qquad \leq C e^{c(\lambda-\epsilon) t_n} \int_1^{\infty} n^{-1} e^{-cy} (\cos(\delta')^{-2} y^2 + (\lambda-\epsilon)^2 t_n^2) y^{-1} \, dy \\
        &\qquad \leq C e^{c(\lambda-\epsilon) t_n} t_n^{-1} \Delta t,
    \end{align*}
    with $y = sn$. For the final integral, we have by \eqref{eq:inequality-along-R-infty},
    \begin{align*}
        &\int_{\gamma_R^{\infty}} \vert F_n(z-(\lambda-\epsilon) \Delta t) \vert \Vert (z - \Delta t (\lambda - \epsilon + A_h))^{-1} \Vert_{L(X_h)} \, \vert dz \vert \\
        &\qquad \leq C e^{c(\lambda-\epsilon) t_n} \int_{\gamma_R^{\infty}} \bigg( ( 1 + c' \cos(\delta') \vert z \vert )^{-n} + e^{-\cos(\delta') n \vert z \vert } \bigg) M' \vert z \vert^{-1} \, \vert dz \vert \\
        &\qquad \leq \frac{C e^{c(\lambda-\epsilon) t_n}}{n-1} \\
        &\qquad \leq C e^{c(\lambda-\epsilon) t_n} t_n^{-1} \Delta t,
    \end{align*}
    as $n > 1$, and $t_n/t_{n-1} \leq 2$, which concludes this proof.
\end{proof}

\begin{proof}[Proof of Lemma \ref{lemma:fully-discrete-inequalities-4}]
    Let $\epsilon, \delta'$ and $M'$ be as in Lemma \ref{lemma:shifted-discrete}. Note that the inequality of this lemma holds by Lemma \ref{lemma:interpolation-inequality} combined with Lemma \ref{lemma:fully-discrete-inequalities-1}, if we are able to show that,
    \begin{align*}
        \Vert (S_h(t_n) - S_{h, \Delta t}(t_n) ) (\lambda + A_h) \Vert_{L(X_h)} \leq C e^{c (\lambda-\epsilon) t_n} t_n^{-1}.
    \end{align*}
    By Lemma \ref{lemma:discrete-semigroup-smoothing}, it suffices to show that,
    \begin{align*}
        \Vert S_{h, \Delta t}(t_n) (\lambda + A_h) \Vert_{L(X_h)} \leq C e^{c(\lambda-\epsilon) t_n} t_n^{-1}.
    \end{align*}
    For $n = 1$, we have,
    \begin{align*}
        (I + \Delta t A_h)^{-1} (\lambda + A_h) &= -\Delta t^{-1} (\lambda - \Delta t^{-1} - (\lambda + A_h))^{-1} (\lambda + A_h) \\
        &= -\Delta t^{-1} (-I + (\lambda - \Delta t^{-1}) (\lambda - \Delta t^{-1} - (\lambda + A_h))^{-1} ).
    \end{align*}
    By using condition \ref{cond:discrete-sectorial}, we therefore get,
    \begin{align*}
        \Vert (I + \Delta t A_h)^{-1} (\lambda + A_h) \Vert_{L(X_h)} \leq C \Delta t^{-1} = C t_1^{-1}. 
    \end{align*}
    
    For $n > 1$, arguing as in the proof of Lemma \ref{lemma:fully-discrete-inequalities-1}, \ref{lemma:fully-discrete-inequalities-2} and \ref{lemma:fully-discrete-inequalities-3}, we have
    \begin{align*}
        \Vert S_{h, \Delta t}(t_n) (\lambda + A_h) \Vert_{L(X_h)} &= \Vert \frac{1}{2 \pi i} \int_{\gamma} (1 + \Delta t z - (\lambda-\epsilon) \Delta t)^{-n} (z+\epsilon) (z - (\lambda - \epsilon + A_h))^{-1} \, dz \Vert_{L(X_h)} \\
        &\leq \frac{1}{2\pi} \bigg( \int_{\gamma_R} + \int_{\gamma_R^{\infty}} \bigg) \vert (1 + \Delta t z - (\lambda-\epsilon) \Delta t)^{-n} (z+\epsilon) \vert \\
        &\qquad \times \Vert (z - (\lambda - \epsilon + A_h))^{-1} \Vert_{L(X_h)} \, \vert dz \vert,
    \end{align*}
    where $\gamma_R = \{ R e^{i \theta}, \ \theta \in [\delta', \pi] \}$, and $\gamma_R^{\infty} = \{ s e^{\pm i \delta'}, \ s \geq R \}$. Setting $R = \lambda$ we have by \eqref{eq:inequality-along-gamma-r}
    \begin{align*}
        &\int_{\gamma_R} \vert (1 + \Delta t z - (\lambda-\epsilon) \Delta t)^{-n} (z+\epsilon) \vert \Vert (z - (\lambda - \epsilon + A_h))^{-1} \Vert_{L(X_h)} \, \vert dz \vert \\
        &\qquad \leq C e^{(c \lambda - c' \epsilon) t_n}.
    \end{align*}
    Further,
    \begin{align*}
        &\int_{\gamma_R^{\infty}} \vert (1 + \Delta t z - (\lambda-\epsilon) \Delta t)^{-n} (z+\epsilon) \vert \Vert (z - (\lambda - \epsilon + A_h))^{-1} \Vert_{L(X_h)} \, \vert dz \vert \\
        &\qquad \leq \int_{\gamma_R^{\infty}} (1 + \Delta t \cos(\delta') \vert z \vert - (\lambda-\epsilon) \Delta t )^{-n} M' \, \vert dz \vert \\
        &\qquad \leq C \int_{\lambda}^{\infty} (1 + \Delta t \cos(\delta) s - (\lambda-\epsilon) \Delta t)^{-n} \, ds \\
        &\qquad \leq \frac{C}{n-1} (1 + \Delta t \cos(\delta) \lambda - (\lambda-\epsilon) \Delta t)^{-n + 1} \Delta t^{-1} \\
        &\qquad \leq C e^{(c \lambda - c'\epsilon) t_n} t_n^{-1},
    \end{align*}
    by similar arguments as in (the last part of) the proof of Lemma \ref{lemma:fully-discrete-inequalities-3}. By redefining $\epsilon > 0$ we can get an exponent of the form $c(\lambda - \epsilon)$. This finishes the proof.
\end{proof}

\begin{proof}[Proof of Lemma \ref{lemma:fully-discrete-inequalities-5}]
    From the proof of Lemma \ref{lemma:fully-discrete-inequalities-4}, we have
    \begin{align*}
        \Vert (S_h(t_n) - S_{h,\Delta t}(t_n)) (\lambda + A_h) \Vert_{L(X_h)} \leq C e^{c(\lambda-\epsilon) t_n} t_n^{-1}.
    \end{align*}
    By combining this inequality with the inequality of Lemma \ref{lemma:fully-discrete-inequalities-3} and the second interpolation inequality of Lemma \ref{lemma:interpolation-inequality}, we have for $\alpha \in [0,1]$ (in particular $\alpha = 1/2$),
    \begin{align*}
        \Vert (S_h(t_n) - S_{h,\Delta t}(t_n)) (\lambda + A_h)^{\alpha} \Vert_{L(X_h)} \leq C e^{c(\lambda-\epsilon) t_n} t_n^{-1} \Delta t^{1-\alpha}.
    \end{align*}
    This finishes the proof. 
\end{proof}
\section{Differential operators on surfaces}\label{app:hypersurfaces}

By a connected and bounded $d$-dimensional $C^k$-surface, we understand:
\begin{enumerate}
    \item a connected and bounded $d$-dimensional hypersurface $\Gamma \subseteq \mathbb{R}^{d+1}$, where $\Gamma$ is open if $\Gamma$ has boundary, and closed otherwise, 
    \item it has a unit normal $\nu : \Gamma \to \mathbb{R}^{d+1}$, an associated tubular neighbourhood, $\mathcal{N}_{\epsilon} := \{ x = x_0 + \alpha \nu(x_0), \ x_0 \in \Gamma, \ \vert \alpha \vert < \epsilon \}$, and
    \item a signed distance function, $d \in C^k(\mathcal{N}_{\epsilon})$, satisfying: $\vert d(x) \vert = \inf_{y\in\Gamma}\vert x-y\vert$, and $\nabla d(x_0 + \alpha \nu(x_0)) = \nu(x_0)$, for any $x_0 \in \Gamma, \ \vert \alpha \vert < \epsilon$. 
\end{enumerate}
The tangent plane at every $x \in \Gamma$, denoted $T_x \Gamma$, is defined as the subspace of $\mathbb{R}^{d+1}$, orthogonal to the surface normal $\nu(x)$, and the tangential derivative of a function $f : \Gamma \to \mathbb{R}$, denoted $\nabla_{\Gamma} f (x) \in T_x \Gamma$, is defined by projecting the gradient in $\mathbb{R}^{d+1}$ of a smooth extension, $\Tilde{f} \in C^1(\mathcal{N}_{\epsilon})$ with $\Tilde{f}\vert_{\Gamma} = f$, onto the tangent plane. That is, we set,
\begin{align}\label{tangential-derivative-definition}
    \nabla_{\Gamma} f (x) = (\underline{\partial}_1 f, \dots, \underline{\partial}_{d+1} f) := (I - \nu(x) \nu(x)^T ) \nabla \tilde{f} (x), \quad x \in \Gamma.
\end{align}

To define higher order derivatives, we proceed as follows: since $\nu$ has a $C^{k-1}(\mathcal{N}_{\epsilon})$-extension (take $\nabla d$), the $k$'th row of the matrix $I - \nu \nu^T$ has a $C^{k-1}(\mathcal{N}_{\epsilon})$-extension, which we denote by $a_k \in C^{k-1}(\mathcal{N}_{\epsilon} ; \mathbb{R}^{d+1})$. For $\tilde{f} \in C^m(\mathcal{N}_{\epsilon})$ with $m \geq 2$, and $f := \tilde{f} \vert_{\Gamma}$, we can compute second order derivatives as 
\begin{align*}
    \underline{\partial}_{i} \underline{\partial}_{j} f (x) := (a_i(x) \cdot \nabla) (a_j(x) \cdot \nabla) \Tilde{f}(x), \quad x \in \Gamma, \quad i,j = 1, \dots, d+1,
\end{align*}
and similarly for higher order. The second order tangential derivative above has an extension to $C^n(\mathcal{N}_{\epsilon})$, where $n = \min(k-2, m - 2)$. For multi-indices $\alpha \in \mathbb{N}_0^{d+1}$, $\vert \alpha \vert \leq \min(k,m)$, and $\beta \in \mathbb{N}_0^{d+1}$ with $\beta_j \in \{ 1, \dots, d+1 \}$ and unique entries, an arbitrary tangential derivative of order $\vert \alpha \vert$ has the representation
\begin{align*}
    \underline{\partial}^{\alpha, \beta} f := \underline{\partial}_{\beta_1}^{\alpha_1} \cdots \underline{\partial}_{\beta_{d+1}}^{\alpha_{d+1}} f, \quad x \in \Gamma.
\end{align*}
Here, $\beta$ was included since the components of tangential derivatives does not commute.

We define $L^p(\Gamma), p \geq 1$ as the linear space of equivalence classes of measurable functions with,
\begin{align*}
    \Vert f\Vert_{L^p(\Gamma)}^p := \int_{\Gamma} \vert f\vert^p \, d\sigma < \infty,
\end{align*}
that are equal if they differ on sets of measure $0$, where $\sigma$ is the surface measure of $\Gamma$. For an $f \in L^p(\Gamma)$, we say that $g \in L^p(\Gamma)$, is a weak derivative if there exists a sequence of functions, $\Tilde{f}_n \in C^s(\mathcal{N}_{\epsilon})$, such that $\tilde{f}_n \vert_{\Gamma} \to f$ in $L^p(\Gamma)$, and $\underline{\partial}_{\beta}^{\alpha}(\tilde{f}_n\vert_{\Gamma}) \to g$ in $L^p(\Gamma)$, where $\sum_i \vert \alpha_i \vert = s$. 

Finally, we define $H^s(\Gamma)$ $s = 1, 2, \dots$ as the subspace of $L^2(\Gamma)$, (where we identify members and their weak derivatives $\sigma$-almost everywhere), that has all weak derivatives up to order $s$ in $L^2(\Gamma)$, and give it the norm,
\begin{align*}
    \Vert f \Vert_{H^s(\Gamma)}^2 := \Vert f \Vert_{L^2(\Gamma)}^2 + \sum_{\vert \alpha \vert \leq s, \beta} \Vert \underline{\partial}_{\beta}^{\alpha} f \Vert_{L^2(\Gamma)}^2. 
\end{align*}
$H^s(\Gamma)$ as defined above is complete, as it is the completion of $C^s(\mathcal{N}_{\epsilon})$ restricted to $\Gamma$, using the $H^s(\Gamma)$-norm. It is a Hilbert space when given the inner product,
\begin{align*}
    (f, g)_{H^s(\Gamma)} := (f, g)_{L^2(\Gamma)} + \sum_{\vert \alpha \vert \leq s, \beta} ( \underline{\partial}_{\beta}^{\alpha} f, \underline{\partial}_{\beta}^{\alpha} g)_{L^2(\Gamma)}.
\end{align*}

We define the Laplace--Beltrami operator for $u \in H^2(\Gamma)$ by $\Delta_{\Gamma} u := \nabla_{\Gamma} \cdot \nabla_{\Gamma} u$. The following lemma is useful for simplifying the expression for the $H^2(\Gamma)$-norm above.
\begin{lemma}[Lemma 3.2 in \cite{2013-elliott}]\label{lemma:cross-terms}
    Let $\Gamma \in C^2$ be compact, and $u \in H^2(\Gamma)$. Then,
    \begin{align*}
        \Vert u \Vert_{H^2(\Gamma)}^2 \leq C \Vert u \Vert_{H^1(\Gamma)} + \Vert \Delta_{\Gamma} u \Vert_{L^2(\Gamma)}^2,
    \end{align*}
    for some $C > 0$.
\end{lemma}
By the divergence theorem and a density argument, we have for $\Gamma \in C^2$ compact (and therefore boundaryless),
\begin{align*}
    \int_{\Gamma} \Delta_{\Gamma} u v \, d \sigma = - \int_{\Gamma} \nabla_{\Gamma} u \cdot \nabla_{\Gamma} v \, d \sigma, 
\end{align*}
for any $u \in H^2(\Gamma)$ and $v \in H^1(\Gamma)$. This observation combined with Lemma \ref{lemma:cross-terms} yields the following estimate. 
\begin{lemma}\label{lemma:elliptic-regularity}
    Let $\Gamma \in C^2$ be compact, and $u \in H^2(\Gamma)$. Then,
    \begin{align*}
        \Vert u \Vert_{H^2(\Gamma)} \leq C \Vert (I - \Delta_{\Gamma}) u \Vert_{L^2(\Gamma)},
    \end{align*}
    for some $C > 0$. 
\end{lemma}
\begin{proof}
    By Lemma \ref{lemma:cross-terms} it suffices to show that,
    \begin{align*}
        \Vert u \Vert_{H^1(\Gamma)}^2 + \Vert \Delta_{\Gamma} u \Vert_{L^2(\Gamma)}^2 \leq C \Vert (I - \Delta_{\Gamma}) u \Vert_{L^2(\Gamma)}^2. 
    \end{align*}
    We have by the divergence theorem,
    \begin{align*}
        \int_{\Gamma} (u - \Delta_{\Gamma} u)^2 \, d\sigma = \int_{\Gamma} u^2 \, d\sigma + 2 \int_{\Gamma} \nabla_{\Gamma} u \cdot \nabla_{\Gamma} u \, d\sigma + \int_{\Gamma} (\Delta_{\Gamma} u)^2 \, d \sigma,
    \end{align*}
    which gives the inequality. 
\end{proof}

\end{document}